\documentclass{amsart}

\usepackage{xypic}
\usepackage{amssymb}
\usepackage{amsmath}
\usepackage{dsfont}

\usepackage{tikz-cd}

\xyoption{all}

\newif\iffurther
\furthertrue

\numberwithin{equation}{section}
\numberwithin{figure}{section}
\theoremstyle{plain}
\newtheorem{thm}{Theorem}[section] 
\newtheorem*{thm*}{Theorem}
\newtheorem{prop}[thm]{Proposition}
\newtheorem{cor}[thm]{Corollary}
\theoremstyle{definition}

\newtheorem{defn}[thm]{Definition}

\newtheorem{lem}[thm]{Lemma}

\theoremstyle{remark}
\newtheorem{rem}[thm]{Remark}

\newtheorem{exmpl}[thm]{Example}

\newtheorem*{acknowledgement*}{Acknowledgement}

\newcommand\suchthat{\;\ifnum\currentgrouptype=16 \middle\fi|\;}

\def\GK{{\operatorname{GKdim}}}

\def\Span{{\operatorname{Span}}}

\def\Aut{{\operatorname{Aut}}}
\def\grAut{{\operatorname{Aut_{gr}}}}
\def\Ann{{\operatorname{Ann}}}
\def\Ker{{\operatorname{Ker}}}

\def\Im{{\operatorname{Im}}}

\def\gr{{\operatorname{gr}}}
\def\sing{{\operatorname{pro}}}
\def\supp{{\operatorname{supp}}}
\def\Subshift{{\operatorname{Subshift}}}
\def\Mon{{\operatorname{Mon}}}

\begin{document}

\title{Noncommutative point spaces of symbolic dynamical systems}



\author{Jason P. Bell}
\address{Department of Pure Mathematics, University of Waterloo, Waterloo, ON, N2L 3G1,
Canada}
\email{jpbell@uwaterloo.ca}

\author{Be'eri Greenfeld}
\address{Department of Mathematics, University of Washington, Seattle, WA, 98195, USA}
\email{grnfld@uw.edu}

\thanks{The first-named author was supported by a Discovery Grant from the National Sciences and Engineering Research Council of Canada.}



\begin{abstract}
We study point modules of monomial algebras associated with symbolic dynamical systems, parametrized by proalgebraic varieties which `linearize' the underlying dynamical systems. Faithful point modules correspond to transitive sub-systems, equivalently, to monomial algebras associated with infinite words. In particular, we prove that the space of point modules of every prime monomial algebra with Hilbert series $1/(1-t)^2$---which is thus thought of as a `monomial $\mathbb{P}^1$'---is isomorphic to a union of a classical projective line with a Cantor set. While there is a continuum of monomial $\mathbb{P}^1$'s with non-equivalent graded module categories, they all share isomorphic parametrizing spaces of point modules. In contrast, free algebras are geometrically rigid, and are characterized up to isomorphism from their spaces of point modules.

Furthermore, we derive enumerative and ring-theoretic consequences from our analysis. In particular, we show that the formal power series counting the irreducible components of the moduli schemes of truncated point modules of finitely presented monomial algebras are rational functions, and classify isomorphisms and automorphisms of projectively simple monomial algebras.
\end{abstract}

\maketitle


\section{Introduction}

Let $X$ be a projective variety over a field $F$:
\[
X = \{[w_0\colon\cdots\colon w_d]\in \mathbb{P}^d\ |\ f_1(w_0,\dots,w_d)=\cdots=f_m(w_0,\dots,w_d)=0\}
\]
for some homogeneous polynomials $f_1,\dots,f_m\in F[x_0,\dots,x_d]$. Then the variety $X$ parametrizes point modules over the homogeneous coordinate ring $A=F[t_0,\dots,t_d]/\left<f_1,\dots,f_m\right>$. A \textit{point module} over a connected graded algebra $A=\bigoplus_{n=0}^{\infty} A_n$ is a graded, cyclic $A$-module $M$ generated in degree $0$ with $1$-dimensional homogeneous components; that is, with Hilbert series $H_M(t)=1+t+t^2+\cdots$. 
This observation provides a way to study noncommutative geometry by analyzing parameter spaces of point modules of noncommutative graded algebras.
Thus, spaces of point modules have become fundamental in noncommutative projective algebraic geometry, enabling one to attach a geometric object to a given noncommutative graded algebra, see \cite{ATV, ATV2, AZ, AZ2, NS, RRZ, RZ, Smith} and references therein. 
Given a connected graded algebra $A$, let $\mathcal{P}_n(A)$ denote the moduli space of $n$-truncations of its point modules; this is a projective scheme. We have a compatible system of morphisms:
$$ \cdots \twoheadrightarrow \mathcal{P}_{n+1}(A)\twoheadrightarrow \mathcal{P}_n(A)\twoheadrightarrow \cdots \twoheadrightarrow \mathcal{P}_0(A) = \{\bullet\} $$
and there is a canonical bijection: 
$$ \mathcal{P}(A) := \varprojlim_{n} \mathcal{P}_n(A) \longleftrightarrow \{\text{Point}\ A\text{-modules}\}/\cong, $$
see \cite{ATV, Rogalski, SV}. The inverse limit is, in general, a proalgebraic variety which is not even an algebraic stack. For instance, the proalgebraic variety associated with a free noncommutative algebra is an infinite product of projective spaces.

The purpose of this paper is to discover and study a symbolic dynamical version of this subject and to decipher the geometry of the spaces of point modules of monomial algebras, building bridges between symbolic dynamics and (pro)algebraic geometry. 
Let $\Sigma$ be a finite alphabet. A \textit{subshift} is a non-empty shift-invariant subset $X\subseteq \Sigma^{\mathbb{N}}$ which is closed in the product topology. With every subshift $X$ one associates a monomial algebra $A_X$, generated by the letters of the alphabet and spanned by all of the finite factors of $X$, which can be thought of as a `ring of functions' on $X$. 
A monomial algebra takes the form $A_X$ if and only if it is \textit{prolongable}; that is, every non-zero monomial can be extended to a longer non-zero monomial.
If $X$ is \textit{transitive}, that is, generated by a single infinite word $w$ then $A_X$ is in fact $A_w$, the monomial algebra of finite subwords of $w$.
We explicitly compute the proalgebraic variety of point modules of $A_X$ by combinatorial means, showing that it can be viewed as a linearization of the underlying subshift (Theorem \ref{decomposition}). While general graded algebras need not admit point modules at all,\footnote{For instance, the prime Noetherian PI-algebra $F+M_2(F)t+M_2(F)t^2+\cdots$ admits no point modules.} infinite-dimensional monomial algebras always do. 
An important role is played by the annihilators of point modules:
\begin{thm}[{Theorem \ref{faithful} and Proposition \ref{sing char}}]
A monomial algebra $A$ takes the form $A_w$ for some infinite word if and only if it admits a faithful point module; and $A$ is prolongable if and only if the intersection of the annihilators of its point modules is zero.
\end{thm}
In particular, the data of point modules of monomial algebras is encoded by their maximal prolongable quotients; these quotients are systematically studied in Section \ref{sec:prolongable radical}.

An important class of monomial algebras are
algebras which can be thought of as monomial---or symbolic dynamical---versions of the homogeneous coordinate ring of $\mathbb{P}^1$:

\begin{thm}[{Theorem \ref{Sturmian_iso}}]
Let $A$ be a prime monomial algebra of Hilbert series $H_A(t)=\frac{1}{(1-t)^2}$. Then the proalgebraic variety of point modules of $A$ is a union of a projective line with a Cantor set, intersecting at two points.
\end{thm}

Thus, the geometry of every `monomial  $\mathbb{P}^1$' is given by a union of a classical $\mathbb{P}^1$ with a copy of the underlying subshift. There is a continuum of pairwise non-isomorphic monomial $\mathbb{P}^1$'s, even with pairwise non-equivalent graded module categories (Proposition \ref{ZhangRigidity}), which share isomorphic proalgebraic varieties of point modules. 
A completely different situation occurs for free algebras, which are geometrically rigid, and more generally, for algebras with irreducible proalgebraic varieties of point modules---there are only countably many such prolongable monomial algebras.
\begin{thm}[{Proposition \ref{Irr} and Theorem \ref{free_rigid}}]
Let $A$ be a prolongable monomial algebra for which $\mathcal{P}(A)$ is irreducible. Then $A$ has a surjective map onto a free algebra, whose kernel is nilpotent. 
If $A$ is a prolongable monomial algebra for which $\mathcal{P}(A)$ is isomorphic to the proalgebraic variety of point modules of a free algebra then $A$ is free.
\end{thm}

A special role in noncommutative algebraic geometry is played by \textit{projectively simple} algebras. These are infinite-dimensional graded algebras all of whose proper graded homomorphic images are finite-dimensional; thus, these are the `minimal' infinite-dimensional graded algebras. Indeed, a monomial algebra is projectively simple if and only if it takes the form $A_X$ for a minimal subshift $X$. An important class of projectively simple algebras are monomial $\mathbb{P}^1$'s,  analyzed in Theorem \ref{Sturmian_iso}. 
A remarkable result of Reichstein, Rogalski and Zhang \cite{RRZ} says that every projectively simple algebra over an algebraically closed field, which admits a point module and which is strongly Noetherian, is finite-codimensional in a twisted homogeneous coordinate rings of a smooth projective variety. 
What can be said about projectively simple graded algebras admitting point modules, but which are not necessarily strongly Noetherian? What can be said about the spaces of point modules of such algebras? These wide fundamental problems are raised in various versions in \cite{RRZ} (see discussion after Theorem 0.4 on page 368 therein and before Example 2.5 on page 376; see also \cite[Remark~4.7]{RZ}) and are partially answered in the current paper for the class of monomial algebras by the aforementioned theorems.

Taking an enumerative approach, we consider the sequence: $$a_n(A) = \#\{\text{Irreducible}~\text{components of}\ \mathcal{P}_n(A)\}.$$
\begin{thm}[{Theorem \ref{rational}}]
Let $A$ be a finitely presented monomial algebra.
Then the formal power series $\sum_{n=0}^{\infty} a_n(A)t^n$ is a rational function.
\end{thm}
A similar result is proved for the (usually bigger) moduli schemes $\mathcal{P}_n(A)\subseteq \overline{\mathcal{P}}_n(A)$ parametrizing graded cyclic modules of Hilbert series $1+t+\cdots+t^n$ (whose limit is $\mathcal{P}(A)$ as well). Irreducible components of both families of moduli spaces have attracted considerable attention (\cite{Smith_Skl,Walton_Skl,Walton_Skl_cor}; see also \cite{Smith}).

We also study the structure of isomorphisms and automorphism groups of monomial algebras:
\begin{thm}[{Theorem \ref{iso mon alg}, Theorem \ref{Aut}, Corollary \ref{cor:iso}}]
Let $A,B$ be monomial algebras. If $A\cong B$ then there is an isomorphism between them given by a permutation on the generators. Furthermore, if $A,B$ are projectively simple then every graded isomorphism between them is given by a permutation and scaling of the generators, and in particular, $\grAut(A)\leq \mathbb{G}_m \wr S_n$.
\end{thm}

Throughout the paper we prove additional results of independent ring-theoretic interest and analyze various examples. 

We start with a background section providing some necessary preliminary notions and results from symbolic dynamics, combinatorics of words and noncommutative algebra.

\textit{Conventions.} In this article, $\mathbb{N}=\{0,1,2,\dots\}$. Graded algebras are by default assumed to be connected $\mathbb{N}$-graded with finite-dimensional homogeneous components and generated in degree $1$. For a graded module $M=\bigoplus_{n\in \mathbb{N}} M_n$ we let $M_{\leq n}=\bigoplus_{i=0}^n A_i,\ M_{\geq n}=\bigoplus_{i=n}^{\infty} M_i$.
Modules are right modules unless stated otherwise.

\section{Growth, monomial algebras, infinite words and subshifts}

\subsection{Growth of algebras}
Let $A$ be a finitely generated associative algebra over a field $F$. Let $V$ be a finite-dimensional generating subspace, namely, $A=F\left<V\right>$. The growth of $A$ with respect to $V$ is the function:
\[ 
\gamma_{A,V}(n)=\dim_F \left(F+V+V^2+\cdots+V^n\right).
\]
If $1\in V$ then equivalently $\gamma_{A,V}(n)=\dim_F V^n$. 

We write $f\preceq g$ if $f(n)\leq g(Cn)$ for some $C>0$ and for all $n\in \mathbb{N}$, and we say that $f$ is asymptotically equivalent to $g$, denoted $f\sim g$, if $f\preceq g$ and $g\preceq f$. We henceforth refer to growth functions of algebras up to asymptotic equivalence. We note that the function $\gamma_{A,V}(n)$ is independent of the choice of the generating subspace up to asymptotic equivalence, and so we write $\gamma_A(n)$ for the growth function of the algebra.

The Gel'fand-Kirillov (GK) dimension of $A$ is:
\[
\GK(A) = \limsup_{n\rightarrow \infty} \frac{\log \gamma_A(n)}{\log n}.
\]
It follows from the definition that $\GK(A)<\infty$ if and only if $\gamma_A(n)$ is polynomially bounded. If $A$ is commutative then $\GK(A)$ coincides with the classical Krull dimension of $A$. The possible values of $\GK(A)$ are $0,1,[2,\infty]$. For more on the growth of algebras and the Gel'fand-Kirillov dimension, see \cite{KL}.

If $A=\bigoplus_{n=0}^{\infty} A_n$ is graded with finite-dimensional homogeneous components then $\gamma_A(n)\sim \dim_F\bigoplus_{i=0}^n A_i$. The Hilbert series of $A$ is the formal power series:
\[
H_A(t) = \sum_{n=0}^{\infty} \dim_F A_n \cdot t^n.
\]
If $A$ is commutative then $H_A(t)$ is a rational function of $t$. Specifically, for the homogeneous coordinate ring of $\mathbb{P}^d$---that is, the polynomial ring $A=F[x_0,\dots,x_d]$---we have:
\[
H_A(t) = \frac{1}{(1-t)^{d+1}}.
\]
A graded algebra is projectively simple if it is infinite-dimensional, but all of its proper graded homomorphic images are finite-dimensional. Every projectively simple algebra is prime.

\subsection{Monomial algebras and infinite words}

Let $\Sigma=\{x_1,\dots,x_m\}$ be a finite alphabet. Let $I\triangleleft F\left<\Sigma\right>$ be an ideal generated by monomials. Then the quotient ring $A=F\left<\Sigma\right>/I$ is called a monomial algebra.
Every monomial algebra $A$ is naturally graded by letting $\deg(x_1)=\cdots=\deg(x_m)=1$, and $A$ decomposes as a direct sum of homogeneous components $A=\bigoplus_{n=0}^{\infty} A_n$ where $A_n$ is spanned by the set of its non-zero length-$n$ monomials, denoted $\mathcal{L}_n(A)$. The set $\mathcal{L}(A):=\bigcup_{n=0}^{\infty} \mathcal{L}_n(A)\subseteq \Sigma^*$, which is a basis for $A$, is a hereditary language over $\Sigma$, namely, a non-empty set of words which is closed under taking subwords. Conversely, for every hereditary language $W\subseteq \Sigma^*$ there exists a monomial algebra:
\[ A_W:=F\left<x_1,\dots,x_m\right>/\left<u\ |\ u\ \text{is not a factor of a word from}\ W\right>
\]
such that $\mathcal{L}(A_W)=W$, and for every monomial algebra $A$, we have that $A_{\mathcal{L}(A)}=A$.
Of special importance is the case where $W$ consists of the subwords (also called factors) of a given (right) infinite word $w\in \Sigma^{\mathbb{N}}$, and we define:
\[ A_w:=F\left<x_1,\dots,x_m\right>/\left<u\ |\ u\ \text{is not a factor of}\ w\right>. 
\]

For a hereditary language $W$ (resp. infinite word $w$) we let $p_W(n)$ (resp. $p_w(n)$) be its complexity function, counting its length-$n$ words. By the above discussion, $p_W(n)=\#\mathcal{L}(A_W)$, so $\gamma_{A_W}(n)=\sum_{i=0}^{n} p_W(i)$.

We say that a monomial algebra $A$ is (right) prolongable if for every non-zero monomial $w\in A$ there exists a non-empty monomial $v\in A$ such that $wv\neq 0$. Similarly, one can define prolongable hereditary languages. 
For every infinite word $w$, the monomial algebra $A_w$ is prolongable. If $A=A_W$ is prolongable then $p_W(n)\leq p_W(n+1)$, and consequently $\gamma_{A_W}(n)\sim np_W(n)$.

\subsection{Dynamical properties of words, subshifts and algebras}

Prolongable algebras are tightly connected to symbolic dynamics. Let $\Sigma= \{x_1,\dots,x_m\}$ be a finite alphabet and consider the (left) shift operator $T\colon \Sigma^{\mathbb{N}}\rightarrow \Sigma^{\mathbb{N}}$ on right infinite words over $\Sigma$.
A (one-sided, right) subshift is a non-empty, shift-invariant closed (in the product topology) subset $X\subseteq \Sigma^{\mathbb{N}}$. We associate with $X$ a monomial algebra $A_X$ which is defined as the monomial algebra corresponding to the language of all finite factors of $X$, denoted $\mathcal{L}(X)$; similarly, one defines the complexity function $p_X(n)$ of $X$. The algebra $A_X$ is always (right) prolongable. Conversely, for any (right) prolongable monomial algebra $A=F\left<\Sigma\right>/I$, let: \[ \Subshift(A)=\{ w=x_{i_0}x_{i_1}\cdots\in \Sigma^{\mathbb{N}}\ |\ \text{all finite factors of}\ w\ \text{are non-zero in}\ A \} \]
be the subshift associated with $A$. The operators $X\mapsto A_X$ and $A\mapsto \Subshift(A)$ on the classes of subshifts and prolongable monomial algebras, are mutual inverses.

Many important dynamical properties of subshifts are reflected in the algebraic structure of their associated monomial algebras, see \cite{BBL,Madill,Nekrashevych}. In particular, a subshift is minimal if it contains no proper subshift; minimality of $X$ is equivalent to projective simplicity of $A_X$. A subshift is eventually periodic if there exist $p,m>0$ such that $T^{p+m}(x)=T^{m}(x)$ for all $x\in X$. The monomial algebra associated with an eventually periodic subshift has linear growth and satisfies a polynomial identity (PI).

We say that a subshift is transitive
if it admits a dense orbit, namely, there exists an infinite word $w\in X$ such that $\{w,T(w),T^2(w),\dots\}$ is dense in $X$. In this case, $A_X=A_w$. Every prime monomial algebra takes the form $A_w$ for some right (or left, or two-sided) infinite word $w$. Thus, $A_w$ is the monomial algebra associated with a $\mathbb{Z}$-subshift (namely, a closed shift-invariant subset of $X\subseteq \Sigma^{\mathbb{Z}}$). Then, one can view $A$ as a `ring of functions' on $X$: as observed in \cite{Nekrashevych}, $A_X$ is isomorphic to a subalgebra of the convolution algebra of the \'etale groupoid $X\rtimes \mathbb{Z}$; specifically, to the subalgebra generated by the sets of germs of the restriction of the shift operator onto the cylindrical
sets $\{w\in X\ :\ w[0] = \sigma\}$, for all $\sigma\in \Sigma$.

An infinite word $w$ is recurrent if every finite factor of it occurs infinitely many times, and uniformly recurrent if for every finite factor $u$ of $w$ there exists some $C_u\in \mathbb{N}$ such that every length-$C_u$ factor of $w$ contains an occurrence of $u$. The monomial algebra $A_w$ is prime if and only if $w$ is recurrent, and projectively simple if and only if $w$ is uniformly recurrent; see \cite{BBL,Madill}.

The following are standard in combinatorics of infinite words (e.g. see \cite{Pir}) and will be freely used throughout the paper.

\begin{lem}[{K\"onig's Lemma}]
Every infinite connected graph whose vertices have finite degrees, contains an infinite path.
\end{lem}
The finite factors of an infinite hereditary language form a tree, with an arrows between words of the form $u\rightarrow ux$ for $s\in \Sigma$. An application of this lemma then yields that every infinite-dimensional monomial algebra has a homomorphic image of the form $A_w$ for some infinite word $w$.


\begin{lem}[{F\"urstenberg's Lemma}]
Let $w$ be an infinite word. Then there exists a uniformly recurrent word $w'$ such that every factor of $w'$ is also a factor of $w$.
\end{lem}
In a ring-theoretic language, every infinite-dimensional monomial algebra has a projectively simple monomial homomorphic image. We conclude with two more lemmas which will be used in the sequel.

\begin{lem} \label{lem:two prolong}
Let $w\in \Sigma^{\mathbb{N}}$ be a non eventually periodic right infinite word. Then for every $n \geq 1$ there exists a length-$n$ factor $u$ of $w$ such that $ux_1,ux_2$ both factor $w$, for some distinct $x_1,x_2\in \Sigma$.
\end{lem}
\begin{proof}
Otherwise, there exists some $n\geq 1$ such that every length-$n$ factor of $w$ has a unique prolongation to the right. Fix some $m_2\geq m_1+n$ such that $w[m_1,m_1+n-1]=w[m_2,m_2+n-1]$. Let $u=w[m_1,m_2-1]$. Then $w=w[0,m_1-1]uuu\cdots$ is eventually periodic.
\end{proof}


\begin{lem} \label{NoIsoPt}
Let $A$ be a prime monomial algebra not of linear growth. Then $\Subshift(A)$ is homeomorphic to the Cantor set.
\end{lem}
\begin{proof}
Every compact, totally disconnected, perfect metric spaces is homeomorphic to the Cantor set, so we have to show that $\Subshift(A)$ has no isolated points.
Since $A$ is prime, $A=A_w$ for some infinite recurrent word $w\in \Sigma^{\mathbb{N}}$. 
perfect, that it, has no isolated point. Pick $u\in \Subshift(A)$ and fix $d\in \mathbb{N}$. We can find $n$ such that $w[n,n+d-1]=u[0,d-1]=:z$. Since $w$ is recurrent, we can find $n'>n$ such that $w[n',n'+d-1]=z$. Therefore, if the factor $z$ has a unique $n'-n$ right prolongation in $w$ then $w$ is eventually periodic and hence $A$ is of linear growth. So there exist at least two occurrences of $z$ with distinct length $n'-n$ prolongations, say, $w[n_1,n_1+d-1]=w[n_2,n_2+d-1]=z$. It follows that either $T^{n_1}(w)$ or $T^{n_2}(w)$ is not $u$, yet has the same length-$d$ prefix as of $u$. Therefore $u$ is not isolated.
\end{proof}

\section{Combinatorial classification of point modules}

\subsection{Infinite words support faithful point modules}

\begin{prop} \label{AW}
Let $w$ be a right-infinite (resp.~left-infinite) word over a finite alphabet. Then the monomial algebra $A_w$ admits a faithful right (resp.~left) point module.
\end{prop}

\begin{proof}
We prove the proposition for right-infinite words.
Let $F\left<x_1,\dots,x_m\right>$ be the free algebra and let $P=P_0+P_1+\cdots$ be the right point $F\left<x_1,\dots,x_m\right>$-module defined as follows:
\begin{eqnarray*}
P_i & = & Fe_i \\
e_i \cdot x_j & = & \delta_{w[i],x_j}e_{i+1}.
\end{eqnarray*}
For all $i\geq 0$ and $1\leq j\leq m$. To see that $P$ is actually an $A_w$-module, let $x_{i_1}\cdots x_{i_k}$ be a word vanishing in $A_w$, namely, it is not a factor of $w$. Then for every $i\geq 0$:
\begin{eqnarray*}
e_i \cdot x_{i_1}\cdots x_{i_k} & = & \delta_{w[i],x_{i_1}}\delta_{w[i+1],x_{i_2}}\cdots \delta_{w[i+k-1],x_{i_k}}e_{i+k} \\ & = & \delta_{w[i,i+k-1],x_{i_1}\cdots x_{i_k}} e_{i+k} = 0.
\end{eqnarray*}
To observe that $P$ is a faithful $A_w$-module, suppose that $0\neq f\in A_w$ annihilates $M$. Since $P$ is graded, we may assume that $f$ is homogeneous, say, of degree $d$. Write $f=\sum_{k=1}^{l} c_k u_k$ where $u_1,\dots,u_l$ are distinct length-$d$ factors of $w$ and the coefficients $c_1,\dots,c_l\in F$ are non-zero. Since $f$ annihilates $P$, it follows that for each $i\geq 0$:
$$0 = e_i \cdot f = e_i \cdot \sum_{k=1}^{l} c_k u_k = e_{i+d}\cdot \left(\sum_{k=1}^{l} c_k\delta_{w[i,i+d-1],u_k}\right).$$
Hence: $$\sum_{k=1}^{l} c_k \delta_{w[i,i+d-1],u_k}=0.$$
Now pick $i$ for which $w[i,i+d-1]=u_1$; this is indeed possible since $u_1$ is a factor of $w$. We now get:
$$0=\sum_{k=1}^{l} c_k \delta_{w[i,i+d-1],u_k}=c_1,$$
contradicting the assumption that all coefficients $c_1,\dots,c_l$ are non-zero. Hence $P$ is a faithful point $A_w$-module.
\end{proof}


\begin{cor}
Every prime monomial algebra admits both right and left faithful point modules. Every infinite-dimensional monomial algebra admits both right and left point modules.
\end{cor}

\begin{proof}
Every prime monomial algebra takes the form $A_w$ for both a right-infinite and a left-infinite word $w$ (see \cite{BBL}), so it admits a faithful point module by Proposition \ref{AW}.
Let $A$ be a finitely generated infinite-dimensional monomial algebra. By F\"urstenberg's lemma, $A$ admits a homomorphic image of the form $A_w$ for some uniformly recurrent word $w$. Therefore $A$ admits a point module by inflation, applying Proposition \ref{AW} to $A_w$.
\end{proof}

Let $A=F\left<x_1,\dots,x_m\right>/I$ be a monomial algebra.
Let $C_0,C_1,\dots\subseteq \{x_1,\dots,x_m\}$ be non-empty subsets. We call the set of infinite words (in the free algebra on $x_1,\dots,x_m$) $\mathcal{T}=C_0C_1\cdots$ a \emph{tree} over $A$ if every finite factor of every infinite word from $\mathcal{T}$ is a non-zero monomial in $A$. Equivalently, $\mathcal{T}$ is a tree over $A$ if and only if there is a well-defined surjection $A\twoheadrightarrow A_\mathcal{T}$ where:
\[ A_\mathcal{T}:=F\left<x_1,\dots,x_m\right>/\left<u\ |\ u\ \text{is not a factor of an infinite word in}\ \mathcal{T}\right> \]
via $x_i\mapsto x_i$.

For an $m$-tuple of infinite sequences of scalars from $F$:
\[ \Lambda = \left( (\lambda_{i,1})_{i=0}^{\infty},\dots,(\lambda_{i,m})_{i=0}^{\infty} \right) \]
we define a collection of infinite words $\mathcal{T}(\Lambda)=C_0C_1\cdots$ where $C_i=\{x_j|\lambda_{i,j}\neq 0\}$.

A point module $P=Fe_0+Fe_1+\cdots$ for which there exists an infinite word $w\in \{x_1,\dots,x_m\}^{\mathbb{N}}$ such that $e_i\cdot x_j=\lambda_{i,j}\delta_{w[i],x_j}e_{i+1}$ for some scalars $\lambda_{i,j}\in F^\times$ will be called a \emph{monomial point module}. Equivalently, this coincides with $\mathcal{T}(\Lambda)$ consisting of a single infinite word (each $C_i$ being a singleton).


\begin{thm} \label{Classification}
Let $A=F\left<x_1,\dots,x_m\right>/I$ be a monomial algebra. Let $P=Fe_0+Fe_1+\cdots$ be a point $A$-module with $e_i\cdot x_j = \lambda_{i,j}e_{i+1}$ for all $i\geq 0,1\leq j\leq m$. Then $\mathcal{T}(\Lambda)$ is a tree over $A$.

Conversely, let: \[ \Lambda = \left( (\lambda_{i,1})_{i=0}^{\infty},\dots,(\lambda_{i,m})_{i=0}^{\infty} \right) \]
be an $m$-tuple of infinite sequences of scalars from $F$ such that $\mathcal{T}(\Lambda)$ is a tree over $A$. Then there is a point $A$-module $P=Fe_0+Fe_1+\cdots$ such that $e_i\cdot x_j = \lambda_{i,j}e_{i+1}$.
\end{thm}

\begin{proof}
Let $A=F\left<x_1,\dots,x_m\right>/I$ be a monomial algebra and let $P=Fe_0+Fe_1+Fe_2+\cdots$ be a point $A$-module. Let $e_i\cdot x_j = \lambda_{i,j}e_{i+1}$ for some scalars $\lambda_{i,j}\in F$. 
For each $i\geq 0$, let $C_i=\{x_j|\lambda_{i,j}\neq 0\}\subseteq \Sigma$ and let: \[ \mathcal{T}=\mathcal{T}(\Lambda)=C_0C_1\cdots \subseteq \{x_1,\dots,x_m\}^{\mathbb{N}}.\]
Since $P$ is a point module (generated by $e_0$), it follows that $C_i\neq \emptyset$ for all $i\geq 0$.
To prove that $\mathcal{T}$ is a tree over $A$ we have to show that given a finite factor of some infinite word from $\mathcal{T}$, say, $x_{r_1}\cdots x_{r_k}$, such that there exists $i\geq 0$ such that: 
\[ 
x_{r_1}\in C_i,\dots,x_{r_k}\in C_{i+k-1}
\]
then $x_{r_1}\cdots x_{r_k}$ is non-zero in $A$. Indeed,
$ \lambda_{i,r_1},\dots,\lambda_{i+k-1,r_k} \neq 0 $ so: \[ e_i\cdot x_{r_1}\cdots x_{r_k} = \lambda_{i,r_1}\lambda_{i+1,r_2}\cdots \lambda_{i+k-1,r_k} e_{i+k} \neq 0 \]
and therefore $x_{r_1}\cdots x_{r_k}$ is a non-zero element of $A$.

Conversely, let: \[ \Lambda = \left( (\lambda_{i,1})_{i=0}^{\infty},\dots,(\lambda_{i,m})_{i=0}^{\infty} \right) \]
be an $m$-tuple of infinite sequences of scalars from $F$ such that $\mathcal{T}=\mathcal{T}(\Lambda)$ is a tree over $A$.
Define a right module over the free algebra $F\left<x_1,\dots,x_m\right>$:
\[ P = Fe_0 + Fe_1 + \cdots \]
via $e_i\cdot x_j = \lambda_{i,j}e_{i+1}$. This is a point module generated by $e_0$ since $\mathcal{T}(\Lambda)=C_0C_1\cdots$ is a tree and therefore each $C_i$ is non-empty. If $x_{r_1}\cdots x_{r_k}$ is not a factor of an infinite word from $\mathcal{T}$ then, for every $i\geq 0$, it does \textit{not} hold that: 
\[ 
x_{r_1}\in C_i,\dots,x_{r_k}\in C_{i+k-1};
\]
equivalently, one of $\lambda_{i,r_1},\dots,\lambda_{i+k-1,r_k}$ is zero. Hence for each $i\geq 0$ it holds that:
\[ e_i\cdot x_{r_1}\cdots x_{r_k} = \lambda_{i,r_1}\lambda_{i+1,r_2}\cdots \lambda_{i+k-1,r_k} e_{i+k} = 0 \]
and therefore $P$ is a well-defined point $A_\mathcal{T}$-module. Since $A\twoheadrightarrow A_\mathcal{T}$ via $x_j\mapsto x_j$, it follows that $P$ is a point $A$-module by inflation.
\end{proof}

\begin{prop} \label{rem}
Every projectively simple monomial algebra which is not of linear growth admits a non-monomial faithful point module.
\end{prop}
\begin{proof}
Let $\Sigma=\{x_1,\dots,x_m\}$ and let $A=F\left<\Sigma\right>/I$ be a projectively simple monomial algebra which is not of linear growth; then $A=A_w$ for some uniformly recurrent left word. Then for every $n \geq 1$ there exists a length-$n$ monomial $0\neq u\in A$ with two distinct left prolongations $x_iu,x_ju\neq 0$ by Lemma \ref{lem:two prolong}. It follows that for some $1\leq i<j\leq m$ there are infinitely many monomials which can be prolonged to the left both by $x_i$ and by $x_j$. By K\"onig's lemma, 
there exists a right infinite word $w\in \Subshift(A)$ such that $x_iw,x_jw\in \Subshift(A)$. Hence we get a tree $\mathcal{T}$ over $A$ with $C_0=\{x_i,x_j\}$. By projective simplicity of $A$, it follows that $A=A_\mathcal{T}$, and the point module $P=Fe_0+Fe_1+\cdots$ given by $e_k\cdot x_l=\delta_{x_l\in C_k}e_{k+1}$ is a non-monomial faithful point $A$-module.
\end{proof}

\subsection{Faithful point modules imply underlying infinite words}

\begin{thm} \label{faithful}
Let $A$ be a finitely generated monomial algebra. Then $A$ admits a faithful right point module if and only if $A$ is of the form $A_w$ for a right infinite word $w$.
\end{thm}
\begin{proof}
That every algebra of the form $A_w$ admits a faithful point module is proved in Proposition \ref{AW}.
Let $\Sigma=\{x_1,\dots,x_m\}$. Let $A=F\left<\Sigma\right>/I$ be a monomial algebra and let $P=Fe_0+Fe_1+Fe_2+\cdots$ be a faithful right $A$-module; clearly $A$ is infinite-dimensional. 
Let $e_i\cdot x_j = \lambda_{i,j}e_{i+1}$ for some scalars $\lambda_{i,j}\in F$. 
For each $i\geq 0$, let $C_i=\{x_j|\lambda_{i,j}\neq 0\}$  and let $\mathcal{T}=C_0C_1\cdots$, which, by Theorem \ref{Classification}, is a tree over $A$. As proved in Theorem \ref{Classification}, the action of $A$ on $P$ factors through $A\twoheadrightarrow A_\mathcal{T}$. Since $P$ is faithful, it follows that $A\cong A_\mathcal{T}$ via the natural isomorphism. Thus, from now on let us refer to $A_\mathcal{T}$ as $A$, naturally identifying their systems of monomial generators with each other.

We may assume that $A$ is \textit{not} prime, since prime monomial algebras take the form $A_w$ for some right-infinite words $w$ (see \cite{BBL}).
For every monomial $u\in F\left<x_1,\dots,x_m\right>$, of length, say, $l$, we let:
\[
S_u := \{n\in \mathbb{N}\ |\ u\in C_n\cdots C_{n+l-1} \}.
\]
Thus $S_u\neq \emptyset$ if and only if $u$ is non-zero in $A$.

\bigskip

\textit{Claim}. If $u,v\in \Sigma^l$ and $S_u=S_v=\{0\}$ then $u=v$.

\textit{Proof of Claim}. Assume that $u,v\in \Sigma^l$ are distinct monomials and $S_u=S_v=\{0\}$. For every non-empty monomial $\rho$, we have that $\rho u=\rho v=0$, for otherwise we would have $|\rho|\in S_u$ or $|\rho|\in S_v$. Since $P$ is generated by $e_0$, for every $i\geq 1$ there exists some non-empty monomial $\rho$ such that $e_0 \cdot \rho = \lambda e_i$ for some $\lambda=\lambda_\rho\neq 0$. Therefore:
\[
e_i \cdot u = e_0 \cdot \lambda^{-1}\rho u = 0
\]
\[
e_i \cdot v = e_0 \cdot \lambda^{-1}\rho v = 0
\]
so every linear combination of $u,v$ annihilates $P_{\geq 1}$.
In addition, $e_0\cdot u = \alpha e_l$ and $e_0\cdot v = \beta e_l$ for some $\alpha,\beta\in F$. If $\alpha=0$ then $P\cdot u=0$, contradicting that $P$ is faithful; otherwise, if $\alpha\neq 0$ then $\beta u - \alpha v\neq 0$ and $P\cdot (\beta u - \alpha v) = 0$, again contradicting the faithfulness of $P$. The claim is proved.

\bigskip

Since we assume that $A$ is non-prime, there exists some $u\in \Sigma^l$ such that $0 < |S_u| < \infty$. Otherwise, for each pair of monomials $u_1,u_2\neq 0$ we can find $i_1\ll i_2$ such that $i_1\in S_{u_1}$ and $i_2\in S_{u_2}$, so $u_1Au_2\neq 0$, contradicting non-primeness. Let $t=\max S_u$; let $u'\in C_0\cdots C_{t+l-1}$ be an arbitrary monomial whose suffix is $u$. Then $0\in S_{u'}$ and if there was $i>0$ such that $i\in S_{u'}$ we would get that $u\in C_{t+i}\cdots C_{t+i+l-1}$, contradicting the maximality of $t$; hence $S_{u'}=\{0\}$. Since every non-zero monomial $u''$ whose prefix is $u'$ also has $S_{u''}=\{0\}$, it follows from the above claim that: 
\[
|C_{t+l}|=|C_{t+l+1}|=\cdots=1.
\]
In particular, 
\[
u'C_{t+l}C_{t+l+1}\cdots
\] 
is a single infinite word, so if $A$ does not take the form $A_w$ for some right infinite word $w$ then there must be some monomial $v\in \mathcal{T}$ which is not a subword of $u'C_{t+l}C_{t+l+1}\cdots$. It follows that $S_v$ is finite (for otherwise it would appear in $C_{t+l}C_{t+l+1}\cdots$). Let $p=\max S_v$, say, $v\in C_p\cdots C_q$ for some $p<q$; extending $v$ to the right if necessary, we may assume that $q\geq t+l-1$. Then extend $v$ to the left, resulting in some $v'\in C_0\cdots C_q$ containing $v$ as a suffix. Thus $S_{v'}=\{0\}$, by the maximality of $p$.
Pick $z\in C_0\cdots C_q$ extending $u'$; hence $S_z=\{0\}$. Now by the above claim, $z=v'$, but then $v$ factors $v'=z\in u'C_{t+l}C_{t+l+1}\cdots$, a contradiction to the way we picked $v$. It follows that $A$ takes the form $A_w$ for some right infinite word, as claimed.
\end{proof}

\begin{cor}
Let $A$ be a finitely generated semiprime monomial algebra. Then $A$ admits a faithful point module if and only if it is prime.
\end{cor}
\begin{proof}
By Theorem \ref{faithful}, $A$ takes the form $A_w$. Since $A$ is semiprime, $w$ is recurrent so $A=A_w$ is in fact prime.
\end{proof}


\begin{cor} \label{base_change}
The existence of a faithful point module of a monomial algebra $A$ is independent of the base field. In particular, $A$ admits a faithful point module if and only if $A\otimes_F K$ does, for every field extension $K/F$.
\end{cor}
\begin{proof}
By Theorem \ref{faithful}, the existence of a faithful point $A$-module is equivalent to $A$ taking the form $A_w$ for some infinite word $w$, which is independent of the base field.
\end{proof}

\begin{rem}
The existence of point modules can be sensitive to field extensions for general (non-monomial) graded algebras. Indeed, let $E/\mathbb{Q}$ be an elliptic curve with no $\mathbb{Q}$-rational points and let $B$ be its homogeneous coordinate ring (with respect to some embedding into a projective space in which it has no rational points). Then $B$ does not have point modules but $B\otimes_\mathbb{Q} \overline{\mathbb{Q}}$ does.
\end{rem}

\section{The prolongable radical} \label{sec:prolongable radical}

\subsection{The maximal prolongable quotient}
Let $A$ be a connected graded algebra, finitely generated in degree $1$. Consider the following ideal:
$$ \sing(A) := \{f\in A\ |\ \exists d\geq 1:\ f\cdot A_d=0 \}. $$
This is the (right) \textit{prolongable radical} of $A$.

\begin{prop} \label{sing is mon} Let $A$ be a connected graded algebra. Then $\sing(A)\triangleleft A$ is a homogeneous ideal, contained in the prime radical of $A$. In particular, if $A$ is infinite-dimensional then $A/\sing(A)$ is infinite-dimensional. Furthermore, if $A$ is a monomial algebra then $\sing(A)$ is generated by monomials.
\end{prop}
\begin{proof}
If $f=f_0+\cdots+f_k\in \sing(A)$ is a sum of homogeneous elements of distinct degrees then for some $d$ we have that $fg=0$ for each $g\in \bigcup_{i=d}^{ \infty} A_i$. It follows that $f_0g+\cdots+f_kg=0$, and since the non-zero summands among $\{f_ig\}_{i=0}^k$ have distinct degrees, $f_0g=\cdots=f_kg=0$ and $f_0,\dots,f_k\in \sing(A)$. 

Next, $\sing(A)$ is contained in the prime radical of $A$ since each $f\in \sing(A)$ generates a nilpotent ideal, since if $f\cdot A_{\geq d}=0$ then $\left< f \right>^{d+1}=0$.

Finally, assume that $A$ is a monomial algebra and let  $f=w_1+\cdots+w_k\in \sing(A)$ be a sum of distinct monomials; we may assume that $w_1,\dots,w_k$ have an equal length by the first part of this proposition. For some $d$, we have that $fu=0$ for every monomial $u$ of length at least $d$, so $w_1u+\cdots+w_ku=fu=0$ is a vanishing sum of distinct monomials, and since $A$ is a monomial algebra it follows that $w_1u=\cdots=w_ku=0$, so $w_1,\dots,w_k\in \sing(A)$. Thus $\sing(A)$ is generated by its monomial elements.
\end{proof}

\begin{lem} \label{prolongable}
A monomial algebra $A$ is prolongable if and only if $\sing(A)=0$.
\end{lem}
\begin{proof}
Clearly $\sing(A)=0$ for every prolongable monomial algebra $A$. If $A$ is non-prolongable, pick a non-zero monomial $w\in A$ which has no right prolongation and observe that $w\cdot A_1=0$ so $0\neq w\in \sing(A)$.
\end{proof}

\begin{lem} \label{lem_prolongable}
Let $A$ be a graded algebra. Then:
$$ \sing(A) \subseteq \bigcap_{\substack{ P\ \text{is a point} \\ \text{right}\  A\text{-module}}} \Ann_A(P), $$
where we take an empty intersection to be $A$.
\end{lem}
\begin{proof}
Let $f\in \sing(A)$ be a homogeneous element of degree $m$, and let $d$ be such that $f\cdot A_d=0$; let $P=Fe_0+Fe_1+\cdots$ be an arbitrary point right $A$-module. Assume to the contrary that $e_i\cdot f\neq 0$ for some $i\geq 0$, say, $e_i\cdot f=\lambda e_{i+m}$ with $\lambda\neq 0$. Then: $$ e_{i+m+d}\in e_{i+m}\cdot A_d=e_i\cdot fA_d=0, $$ a contradiction. Hence $f\in \Ann_A(P)$.
\end{proof}

\begin{prop} \label{sing char}
Let $A$ be a monomial algebra. Then:
$$ \sing(A) = \bigcap_{\substack{ P\ \text{is a point} \\ A\text{-module}}} \Ann_A(P).$$
\end{prop}

\begin{proof}
Let $A=F\left<x_1,\dots,x_m\right>/I$ be a finitely generated infinite-dimensional monomial algebra. The inclusion $\subseteq$ follows from Lemma \ref{lem_prolongable}. Conversely, let $f\notin \sing(A)$. Write $f=\sum_{i=1}^n c_iw_i$, a linear combination of distinct non-zero monomials with non-zero coefficients. At least one of these monomials, say, $w_i$ does not belong to $\sing(A)$. This means that there exist letters $x_{j_1},x_{j_2},\dots$ such that the words:
\begin{align}
    w_ix_{j_1},w_ix_{j_1}x_{j_2},\dots \label{non-zero monomials}
\end{align}
are all non-zero. Consider the ideal $J\triangleleft A$ generated by all of the monomials which do not occur as subwords of the words in the above list (\ref{non-zero monomials}). Then $A/J\cong A_z$ is the monomial algebra associated with the right-infinite word:
$$ z = w_ix_{j_1}x_{j_2}\cdots $$
Since $J$ is a monomial ideal, $f\in J$ if and only if  $w_1,\dots,w_n\in J$; since $w_i\notin J$, also $f\notin J$. However, by Proposition \ref{AW} the algebra $A_z$ admits a faithful point module $P$, which is by inflation a point $A$-module and $\Ann_A(P)=J$. Thus $f\notin J = \Ann_A(P)$ and the claim is proved.
\end{proof}

\begin{rem}
In general, the prolongable radical need not coincide with the intersection of all of the annihilators of its point modules. Indeed, if $A$ is a prime graded algebra without point modules, e.g., $F+M_2(F)t+M_2(F)t^2+\cdots$ then $\sing(A)=0$.
\end{rem}

While the definition of a prolongable monomial algebra is combinatorial in nature, it has an equivalent purely algebraic characterization which follows from the above analysis:

\begin{cor} \label{prolongable cor}
A monomial algebra is prolongable if and only if the intersection of the annihilators of its point modules is trivial.
\end{cor}
\begin{proof}
This is immediate from Lemma \ref{prolongable} and Proposition \ref{sing char}.
\end{proof}

\subsection{Growth of prolongable quotients}

\begin{prop} \label{growth max prolongable subexp}
Let $g\colon \mathbb{N}\rightarrow \mathbb{N}$ be a subexponential function. Then there exists a monomial algebra $A$ whose growth is $\gamma_A\succeq g$ but $A/\sing(A)$ has polynomial growth.
\end{prop}
\begin{proof}
It follows from \cite[Theorem 1.1]{BellZelmanov} that for every subexponential function $g\colon \mathbb{N}\rightarrow \mathbb{N}$ there exists a finitely generated algebra $A$ such that $g(n)\preceq \gamma_A(n)\prec \exp(n)$, and furthermore, $A$ is realized as a monomial algebra $A=F\left<x,y\right>/I$ whose hereditary language $\mathcal{L}(A)$ is:
$$ \bigcup_{n\geq 1} T(d_n,n) \cup \bigcup_{n\geq 1} x^{e_n}T(d_n-1,n-e_n) $$
where: $$ T(d,n) = \{x^n\} \cup \{x^iyx^{a_1}yx^{a_2}\cdots x^{a_s}yx^j\ \text{of length $n$ with}\ a_1,\dots,a_s\geq d \} $$
for some sequence $\{d_n\}_{n \geq 1}$. The sequence $\{d_n\}_{n\geq 1}$ tends to infinity, since the realized growth function $f$ is subexponential (see \cite[Page~691]{BellZelmanov}). Notice also that $y^2=0$. By the structure of $\mathcal{L}(A)$ as above, for every $k$ there exists some $l$ such that every non-zero monomial of $A$ of length greater than $l$ cannot contain any occurrence of $yx^iy$ for any $i\leq k$. Since every monomial in the the ideal $\left<y\right>^2=AyAyA$ contains an occurrence of $yx^iy$ for some $i>0$, this ideal is contained in the prolongable radical, and furthermore, it can be shown that $A/\left<y\right>^2$ is prolongable and has GK-dimension $2$, so $\GK(A/\sing(A))=2$. 
\end{proof}

The case of algebras of exponential growth is very different:

\begin{prop} \label{growth max prolongable exp}
Let $A$ be a graded algebra of exponential growth. Then $A/\sing(A)$ has exponential growth too.
\end{prop}

\begin{proof}
Let $B=A/\sing(A)$, so $B_n = \frac{A_n}{A_n \cap \sing(A)}$, and let $\gamma_A(n)=\dim_F A_{\leq n},\gamma_B(n)=\dim_F B_{\leq n}$ be the growth functions of $A,B$ respectively. Suppose that $\gamma_B(n)$ grows subexponentially. Given $\varepsilon>0$, fix $m\gg_\varepsilon 1$ such that $\gamma_B(m)^{1/m}<1+\varepsilon$. Since $\dim_F A_{\leq m}<\infty$, there exists some $M$ such that for every $u\in A_{\leq m} \cap \sing(A)$ we have that $u\cdot A_M = 0$.
Given $n\geq M+m$, write $n=tm+M+r$ with $t\geq 1$ and $0\leq r<m$ and consider the (surjective) multiplication map:
$$ \mu\colon A_m^{\otimes t} \otimes_F A_{M+r} \twoheadrightarrow A_{n}. $$
Observe that $\mu$ induces a well-defined surjective map:
$$ \overline{\mu}\colon \left(\frac{A_m}{A_m \cap \sing(A)}\right)^{\otimes t} \otimes_F A_{M+r} \twoheadrightarrow A_{n} $$
from which it follows that:
$$ \dim_F A_{n} \leq (\dim_F B_m)^t \cdot \dim_F A_{M+r} $$
thus:
$$ \left(\dim_F A_{n}\right)^{1/n} \leq (\dim_F B_m)^{1/m} \cdot \left(\dim_F A_{M+r}\right)^{1/n} $$
and since $\dim_F A_{M+r}=O_n(1)$ and $(\dim_F B_m)^{1/m}<1+\varepsilon$, it follows that for $n\gg_\varepsilon 1$ we have $\dim_F A_{n}\leq (1+\varepsilon)^n$. Since this can be done for every $\varepsilon>0$, it follows that $\gamma'_A(n)$ grows subexponentially and therefore so does $\gamma_A(n)$.
\end{proof}


\subsection{Left prolongability vs. right prolongability}
In a similar vein to the above definition, one can define the left prolongable radical. For the sake of clarity in this subsection, let us distinguish between the two notions:
\begin{eqnarray*}
\sing_r(A) & = & \{f\in A\ |\ \exists d\geq 1:\ f\cdot A_d=0 \} \\
\sing_l(A) & = & \{f\in A\ |\ \exists d\geq 1:\ A_d \cdot f=0 \}
\end{eqnarray*} 
The analysis from the previous seubsections applies to $\sing_l(A)$ mutatis mutandis. In particular, $\sing_l(A)$ is equal to the intersection of annihilators of left point $A$-modules and $\sing_l(A)$ is contained in the prime radical of $A$.

A natural question is: how do $\sing_r(A)$ and $\sing_l(A)$ interact with each other for an arbitrary infinite-dimensional monomial algebra $A$? Recall that the lowest possible growth rate for $A/\sing_r(A),A/\sing_l(A)$ is linear, both being infinite-dimensional. 



\begin{prop} \label{Bergman's_gap_prolongable}
Let $A$ be an infinite-dimensional monomial algebra. Then the quotient $A/\sing_l(A)$ has linear growth if and only if $A/\sing_r(A)$ does.
\end{prop}


We will need the following combinatorial lemma.

\begin{lem}[{\cite[Lemma~2.4]{KL}}] \label{lem_Bergman}
Let $W$ be a hereditary language over the letters $x_1,\dots,x_m$. If for some $C\geq 1$ we have $p_W(C)\leq C$ then $p_W(n)\leq C^3$ for all $n\geq C$. In particular, $A_W$ has linear growth.
\end{lem}



\begin{proof}[{Proof of Proposition \ref{Bergman's_gap_prolongable}}]
Let $A=F\left<x_1,\dots,x_m\right>/I$ be a monomial algebra that is infinite-dimensional. Assume that $A/\sing_l(A)$ has a linear growth and let us prove that $\overline{A}:=A/\sing_r(A)$ has a linear growth too. 
Since $\overline{A}/\sing_l(\overline{A})$ is a homomorphic image of $A/\sing_l(A)$, it has a linear growth too. Hence without loss of generality we may replace $A$ by $\overline{A}$ and aim to prove the following: if $A$ is right prolongable and $A/\sing_l(A)$ has a linear growth then $A$ has a linear growth.

Assume that $A/\sing_l(A)$ has a linear growth. It follows by Lemma \ref{lem_Bergman} that there exists some constant $C>0$ such that for all $n$:
\begin{eqnarray*}
C  \geq  \dim_F (A/\sing_l(A))_n & = & \#\mathcal{L}_n(A) - \#\left( \mathcal{L}_n(A)\cap \sing_l(A) \right) \\
& = & \#\mathcal{L}_n(A) - \#\{u\in \mathcal{L}_n(A)|\exists d:\ A_d\cdot u=0\}.
\end{eqnarray*}
Fix $n \geq C$. Let $D\geq 1$ be such that if $u\in \mathcal{L}_n(A)\cap \sing_l(A)$ then $A_D\cdot u=0$.
Consider: \[ J = \Span_F \{ u\in \mathcal{L}(A)|A_D\cdot u=0 \} \triangleleft A \] and let $A'=A/J$.
Notice that: \[ \mathcal{L}_n(A') \cap \{u\in \mathcal{L}_n(A)|\exists d:\ A_d\cdot u=0\} = \varnothing \]
so $\#\mathcal{L}_n(A')\leq C \leq n$, and hence by Lemma \ref{lem_Bergman}, for each $i\geq n$ we have: \[\#\mathcal{L}_i(A')\leq n^3 = (C+1)^3.\]

Fix $i > D$. Given a monomial $u\in \mathcal{L}_i(A)$, let $\pi(u)$ denote its length $i-D$ suffix. We claim that $\pi(u)\in \mathcal{L}_{i-D}(A')$. Indeed, if $u\in \mathcal{L}_i(A)$ has $\pi(u)=vtv'$ for some $t\in J$ and $v,v'$ arbitrary monomials then we can write $u=u_0\pi(u)$ with $u_0$ of length $D$, and so $u=u_0vtv'=0$ by definition of $J$.
Therefore, for every $i>D$ we have a map:
\[ \pi\colon \mathcal{L}_i(A) \rightarrow \mathcal{L}_{i-D}(A') \]
such that $\#\pi^{-1}(\xi)\leq m^D$ for each $\xi$ in the codomain.
It follows that for every $i\geq n+D$:
\[ \#\mathcal{L}_i(A) \leq m^D \cdot \#\mathcal{L}_{i-D}(A') \leq m^D(C+1)^3 = O_i(1) \]
and it follows that $A$ itself has linear growth.
\end{proof}

The left-right growth symmetry concluded in Proposition \ref{Bergman's_gap_prolongable} does not hold anymore if we replace the linear growth assumption by, say, a quadratic growth assumption (by Bergman's gap theorem \cite{Bergman_gap} and \cite[Theorem~2.5]{KL}, the slowest possible super-linear growth rate of an algebra is quadratic); this is demonstrated in the following example. Therefore, the case where the maximal prolongable quotients have a linear growth is really special.

\begin{exmpl}
Let $w=x^{a_1}y^{b_1}x^{a_2}y^{b_2}\cdots$ be a right-infinite word over the alphabet $\{x,y\}$ for some sequences $\{a_i\}_{i=1}^{\infty}$ and $\{b_i\}_{i=1}^{\infty}$ and let $A=A_w$. If $a_i,b_i\xrightarrow{i\rightarrow \infty} \infty$ then every monomial in $A$ containing a subword of the form $x y^r x$ or $y x^r y$ for $r\geq 1$ belongs to $\sing_l(A)$. Consequently, $$A/\sing_l(A)\cong F\left<x,y\right>/\left<x y^r x,y x^r y\ |\ r\geq 1\right>=\Span_F \{x^ay^b,y^bx^a\ |\ a,b\geq 0\}$$ which has a quadratic growth.

Let $u_n$ be a word in the free semigroup on $X,Y$ such that all length-$n$ words factor it (e.g. $u_n=W_1\cdots W_{2^n}$ where $\{X,Y\}^n = \{W_1,\dots,W_{2^n}\}$). Let $u_n^{[k]}$ be the word obtained from $u_n$ by replacing each $X$ by $x^k$ and each $Y$ by $y^k$.
Now let: $$ w=u_1^{[1]}u_2^{[2]}u_3^{[3]}\cdots$$ and let $A=A_w$. By the above argument, $A/\sing_l(A)$ has quadratic growth. However, $\sing_r(A)=0$ so $A/\sing_r(A)\cong A$, and for each $n$ we have at least $2^n$ words of length $n^2$ (namely, all words resulting from applying the substitution $X\mapsto x^n,Y\mapsto y^n$ to all length-$n$ words on $X,Y$). Thus: $$ \gamma_{A/\sing_r(A)}(n)=\gamma_A(n)\succeq 2^{\sqrt{n}}. $$ Considering $w=u_1^{[i_1]}u_2^{[i_2]}u_3^{[i_3]}\cdots$ with a modified sequence $i_1,i_2,\dots$ one can make $\gamma_A$ grow faster than an arbitrary subexponential function.
\end{exmpl}

\section{Isomorphisms} \label{sec:isomorphisms}

\subsection{Isomorphic point modules}
A homomorphism $\varphi\colon M\rightarrow M'$ between graded modules is graded of degree $d\in \mathbb{Z}$ if $\varphi(M_i)\subseteq M'_{i+d}$ for every $i\geq 0$. For an element $v$ of a graded module we let $\supp(v)$ be the set of degrees of non-zero components on which $v$ is supported.

\begin{thm} \label{Graded_isomorphisms}
Let $A$ be a connected graded algebra, generated in degree $1$ and let $\varphi\colon M\xrightarrow{\sim} M'$ be an isomorphism of point $A$-modules. Then $\varphi$ is graded of degree $0$.
\end{thm}

\begin{proof}
Let $A_1=Fx_1+\cdots+Fx_m$ generate $A$ as an algebra (along with $A_0=F$).
Fix homogeneous bases: $$ M=Fv_0+Fv_1+\cdots\ \ \text{and}\ \ M'=Fw_0+Fw_1+\cdots $$ 
For each $u\in M'$ let $L(u)=\#\supp(u)$, namely, if $u=c_1w_{j_1}+\cdots +c_tw_{j_t}$ with all $j_1,\dots,j_t$ distinct and $c_1,\dots,c_t$ non-zero then $L(u)=t$.
We claim that: 
\begin{align}
L(\varphi(v_1)) \geq L(\varphi(v_2)) \geq \cdots    \label{L non inc}
\end{align}
Indeed, for each $i\geq 0$ we can find $1\leq j\leq m$ such that $v_i\cdot x_j=cv_{i+1}$ for some $c\neq 0$, so $\varphi(v_{i+1})=c^{-1}\varphi(v_i)\cdot x_j$ and therefore $L(\varphi(v_{i+1}))\leq L(\varphi(v_i))$. Let: $$ d=\min \{L(\varphi(v_i))\ |\ i\geq 0\}\geq 1. $$ 
(Indeed, $d\geq 1$ since if $\varphi(v_i)=0$ for some $i$ then $\dim_F \Im(\varphi)<\infty$.)
We first claim that in fact $L(\varphi(v_i))=d$ for all $i\geq 0$, namely, the sequence (\ref{L non inc}) is constant.
Otherwise, let $i$ be the maximum index such that $L(\varphi(v_i))=e>d$ and write: $$ \varphi(v_i)=c_1w_{j_1}+\cdots+ c_ew_{j_e} $$ for $j_1<\cdots<j_e$ and $c_1,\dots,c_e\neq 0$. 
Notice that $\varphi(v_{i+1})$ is supported on $d<e$ homogeneous components which are contained in $\{j_1+1,\dots,j_e+1\}$, since $v_{i+1}=v_i\cdot g$ for some $g\in A_1$; pick $j_l$ (for some $1\leq l\leq e$) such that $j_l+1\notin \supp(\varphi(v_{i+1}))$. Then for every $a\in A_1$ we have: $$ \varphi(v_i\cdot a)=c_1w_{j_1}\cdot a+\cdots+c_ew_{j_e}\cdot a=c'_1w_{j_1+1}+\cdots+c'_ew_{j_e+1} $$ for some (possibly zero) scalars $c'_1,\dots,c'_e$. Since $v_i\cdot a\in M_{i+1}=Fv_{i+1}$, it follows that $\supp(\varphi(v_i\cdot a))\subseteq \supp(\varphi(v_{i+1}))$, so $w_{j_l}\cdot a=0$. It follows that $w_{j_l}\cdot A_{\geq 1}$, contradicting that $M'$ is generated by $w_0$.
Hence for all $i\geq 0$ we have that $L(\varphi(v_i))=d$. We have thus established that (\ref{L non inc}) is constant.

Now let us furthermore prove that $d=1$. Assume otherwise to the contrary, so $L(\varphi(v_0))>1$. Let $\varphi(v_0)=\alpha_0w_0+\cdots+\alpha_pw_p$ with $\alpha_p\neq 0$ and $p>0$; furthermore, notice that $\alpha_0\neq 0$ since $\varphi$ is surjective, and if $\alpha_0=0$ then $\Im(\varphi)\subseteq M'_{\geq 1}=\sum_{i=1}^{\infty} Fw_i$.
Since $\varphi$ is surjective, there exists a linear combination: $$ \xi=c_0v_0+\cdots+c_rv_r $$ such that $\varphi(\xi)=w_0$.
Notice that $c_0\neq 0$, since otherwise $\varphi\left(M_{\geq 1}\right)=M'$, contradicting that $\varphi$ is injective. Moreover, $c_r\neq 0$ and $r>0$, since $d>1$ (so it is impossible that $\varphi(c_0v_0)=w_0$). 
Pick a homogeneous element $a\in A_r$ such that $v_0\cdot a=v_r$. Now: $$ \varphi(c_rv_r)=c_r\alpha_0w_0\cdot a+\cdots+c_r\alpha_pw_p\cdot a. $$ 

Since $c_r\neq 0$, it follows that $\varphi(c_rv_r)$ is non-zero and supported on the same number of homogeneous components as $\varphi(v_0)$; it then holds that $w_p\cdot a=\beta w_{p+r}$ for some $\beta\neq 0$. Otherwise, if $\beta=0$ then $L(\varphi(c_rv_r))<L(\varphi(v_0)$.
Observe that $\varphi(v_0),\dots,\varphi(v_{r-1})\in Fw_0+\cdots+Fw_{p+r-1}$. Hence: 
\begin{eqnarray*}
w_0=\varphi(\xi) & = & c_0\varphi(v_0)+\cdots+c_r\varphi(v_r) \\ & = & \beta_0 w_0+\cdots+\beta_{p+r-1}w_{r-1}+c_r\alpha_p\beta w_{p+r}
\end{eqnarray*}
for some $\beta_0,\dots,\beta_{r-1}\in F$ and with $c_r,\alpha_p,\beta\neq 0$, yielding a non-trivial linear combination of homogeneous basis elements, a contradiction.

Therefore, for every $i\geq 0$ we have that $L(\varphi(v_i))=1$. Write $\varphi(v_0)=\gamma_0v_t$ for some $t\geq 0$ and $\gamma_0\in F$. Then $\varphi(v_i)=\gamma_iv_{i+t}$ for some $\gamma_i\in F$. Now $\Im(\varphi)=\sum_{i=t}^{\infty} Fw_i=M'_{\geq t}$ and since $\varphi$ is surjective, $t=0$. Hence $\varphi$ is graded of degree $0$.
\end{proof}

Theorem \ref{Graded_isomorphisms} cannot be improved to non-point modules. Here are two quick examples.

\begin{exmpl}
Let $M=M_0+M_1+\cdots$ be the graded $F[t]$-module given by $M_i=Fe_i+Ff_i$ and $e_i\cdot t=e_{i+1},\ f_i\cdot t=f_{i+1}$. Then $M$ is a fat point module (generated by $M_0$, with Hilbert series $2(1-t)^{-1}$) and every matrix in $\text{GL}_2(F[t])\setminus \text{GL}_2(F)$ induces a non-graded automorphism of $M$.
For a concrete example, take $\varphi\colon M\rightarrow M$ given by $\varphi(e_0)=e_0+e_1+f_1,\ \varphi(f_0)=e_0+f_0$.
\end{exmpl}

\begin{exmpl}
Let $A=F+A_1+A_2+\cdots$ where $A_{\geq 1}$ is a nil graded algebra (e.g. a nil Golod-Shafarevich algebra). Then every $f\in A\setminus F$ with a non-zero degree-$0$ term is invertible and induces a non-graded automorphism $a\mapsto fa$ of $A$ as a right module over itself.
\end{exmpl}

\subsection{Isomorphic monomial algebras} \label{subsec:iso mon alg}
Isomorphisms between monomial algebras can be quite complicated, for instance, for free algebras \cite{Umirbaev}.
For a permutation $\sigma\in S_m$ define an automorphism of the $m$-generated free algebra $$\varphi_\sigma\colon F\left<x_1,\dots,x_m\right>\rightarrow F\left<x_1,\dots,x_m\right>$$ by $\varphi_\sigma(x_i)=x_{\sigma(i)}$.

\begin{thm}\label{iso mon alg}
Let $A,B$ be monomial algebras and fix monomial presentations: \[ A=F\left<x_1,\dots,x_m\right>/I,\ B=F\left<x_1,\dots,x_{m'}\right>/J \] with no redundant $x_i$'s, namely, $x_1,\dots,x_m\notin I$ and $x_1,\dots,x_{m'}\notin J$. Suppose that $A\cong B$ as (abstract) algebras. Then $m=m'$ and there exists a permutation $\sigma\in S_m$ such that $J=\varphi_\sigma(I)$.
\end{thm}

This means that isomorphic monomial algebras are equal to each other up to renaming the monomial generators.

\begin{proof}[{Proof of Theorem \ref{iso mon alg}}]
By \cite[Theorem~1]{BellZhang} there is a graded isomorphism $f\colon A\rightarrow B$ which induces a bijective linear map on the degree-$1$ components, $f(A_1)=B_1$. Thus $m=\dim_F A_1=\dim_F B_1=m'$ (notice that since $I,J$ are monomial ideals and $x_1,\dots,x_m\notin I,\ x_1,\dots,x_{m'}\notin J$ then these two sets are bases of $A,B$ respectively).

For each $1\leq i\leq m$ let $S_i\subseteq \{1,\dots,m\}$ be the index set of the support of $f(x_i)$ and write $f(x_i)=\sum_{j\in S_i} c_{ij}x_j$ with all $c_{ij}\neq 0$. For every $1\leq i_1<\cdots<i_t\leq m$ we have:
\[
f(x_{i_1}),\dots,f(x_{i_t})\in \Span_F \{x_j|j\in S_{i_1}\cup\cdots\cup S_{i_t}\} 
\]
and since $x_{i_1},\dots,x_{i_t}$ are linearly independent and $f$ is an isomorphism, it follows that: \[
|S_{i_1}\cup\cdots \cup S_{i_t}|\geq t.
\] By Hall's marriage theorem there exists a permutation $\sigma\in S_m$ such that $\sigma(i)\in S_i$ for all $1\leq i\leq m$. 
Given $x_{i_1}\cdots x_{i_k}\in I$, let us show that $x_{\sigma(i_1)}\cdots x_{\sigma(i_k)}\in J$. Indeed, since $f$ is a homomorphism it follows that: 
\begin{eqnarray*} f(x_{i_1}\cdots x_{i_k}) & = & \left(\sum_{j_1\in S_{i_1}} c_{i_1 j_1}x_{j_1}\right)\cdots \left(\sum_{j_k\in S_{i_k}} c_{i_k j_k}x_{j_k}\right) \\
& = & \sum_{\substack{j_1\in S_{i_1},\dots,j_k\in S_{i_k}}} c_{i_1 j_1}\cdots c_{i_k j_k} x_{j_1}\cdots x_{j_k}
\end{eqnarray*}
belongs to $J$. Since $J$ is a monomial ideal and the above sum is a linear combination of monomials which are distinct in the free algebra, with non-zero coefficients, it follows that each one of them belongs to $J$. In particular, $x_{\sigma(i_1)}\cdots x_{\sigma(i_k)}\in J$ as $\sigma(i_l)\in S_{i_l}$ for each $1\leq l\leq k$. Therefore $\varphi_\sigma(I)\subseteq J$. Interchanging the roles of $A$ and $B$, we get that $\varphi_\pi(J)\subseteq I$ for some $\pi\in S_m$. It follows that $\varphi_{\pi\sigma}(I)\subseteq I$ and thus by induction $\varphi_{\pi \sigma}^d(I)\subseteq \varphi_{\pi\sigma}(I)\subseteq I$ for every $d\geq 1$; taking $d=n!$ gives: $$ I=\varphi^{n!}_{\pi\sigma}(I)\subseteq \varphi_{\pi\sigma}(I)\subseteq I, $$ so $\varphi_{\pi \sigma}(I)=I$. For every $a\in J$ we have that $\varphi_\pi(a)\in I=\varphi_{\pi \sigma}(I)=\varphi_{\pi}\varphi_{\sigma}(I)$, so $a\in \varphi_\sigma(I)$. It follows that $J=\varphi_\sigma(I)$, as claimed.
\end{proof}

Theorem \ref{iso mon alg} also illuminates a major difference between prolongable monomial algebras and their corresponding subshifts. While the assignment $A\mapsto \Subshift(A)$ is formally one-to-one, 
it is far from being so when considered in the relevant categories, of graded algebras and subshifts, respectively. 
Let: \[ {\bf t}=01101001\cdots \] be the Thue-Morse binary word. This word can be defined as the fixed point in $\{0,1\}^{\mathbb{N}}$ of the map induced by the endomorphism $\phi \colon \{0,1\}^*\to \{0,1\}^*$ given by $\phi(0)=01$ and $\phi(1)=10$. Namely, ${\bf t} = \phi^{\omega}(0) := \lim_{n\rightarrow \infty} \phi^n(0)$. For more on the Thue-Morse sequence see \cite{AS}.

We create infinite words ${\bf w}$ and ${\bf w}'$ over the alphabet $\{x,w,z\}$ by applying respectively the codings $\mu,\mu'\colon \{0,1\}\to \{x,w,z\}^*$ to ${\bf t}$, where: 
\[ \mu(0)=xw^2, \mu(1)=xz^2\ \ \text{and}\ \  \mu'(0)=xwz, \mu'(1)=xz^2.\] Then ${\bf w}$ and ${\bf w}'$ are uniformly recurrent as ${\bf t}$ is. Let $A_{\bf w}$ and $A_{{\bf w}'}$ denote the monomial algebras on generators $x,w,z$ in which a monomial is zero in $A_{\bf w}$ (respectively $A_{{\bf w}'}$) if it is not a subword of ${\bf w}$ (respectively ${\bf w}'$).


The algebras $A_{\bf w}$ and $A_{{\bf w}'}$ are non-isomorphic projectively simple monomial algebras both of quadratic growth. 
Indeed, since ${\bf w}$ and ${\bf w}'$ are uniformly recurrent, it follows that $A_{\bf w}$ and $A_{{\bf w}'}$ are projectively simple monomial algebras. It can be seen that $A_{\bf w} \not\cong A_{{\bf w}'}$ by Theorem \ref{iso mon alg}, since $x^2=w^2=0$ in $A_{{\bf w'}}$ while only one monomial generator of $A_{\bf w}$ is squared zero.

However, the subshifts generated by ${\bf w}$ and by ${\bf w}'$ are conjugate to each other.
To see this, we let 
$f\colon \{x,w,z\}^*\to \{x,w,z\}^*$ denote the involution that takes a word $u$ and changes all occurrences of $xw^2$ to $xwz$ and all occurrences of $xwz$ to $xw^2$.  Note that since there can be no overlaps between subwords of the form $xw^2$ and $xwz$, the map $f$ is defined unambiguously and extends to infinite words. In symbolic dynamical terms, $f$ induces a sliding block code with memory $2$ and zero anticipation, induced by the block map $\Phi\colon \{x,w,z\}^3\rightarrow \{x,w,z\}$ given by: \[ \Phi(xww)=z,\Phi(xwz)=w\ \text{and}\ \Phi(v_1v_2v_3)=v_3\ \text{for}\ v_1v_2v_3\neq xww,xwz. \] Since $f({\bf w})={\bf w}'$ and $f^2=\text{id}$ (hence bijective), it is a conjugacy between the subshifts generated by ${\bf w}$ and ${\bf w}'$.

\subsection{Automorphisms}

Let $A$ be a connected graded algebra. Let $\Aut(A)$ be its automorphism group and let $\grAut(A)$ its subgroup of graded automorphisms. Graded automorphisms play a major role in  twisting graded algebras \cite{ZhangTwists}; this will be revisited in the sequel.

Let $A=F\left<x_1,\dots,x_n\right>/I$ be a monomial algebra, and assume that this is a non-redundant presentation, namely, no $x_i$ belongs to $I$. Then the torus $\mathbb{G}_m^{\times n}$  acts faithfully on $A_1=Fx_1+\cdots+Fx_n$ and thus embeds into $\grAut(A)$. In general, $\grAut(A)$ can be very big, e.g. for free algebras we have $\grAut(F\left<x_1,\dots,x_n\right>)\cong \text{GL}_n(F)$ (and the non-graded automorphism group is even wilder as mentioned before, see \cite{Umirbaev}).
\begin{thm} \label{Aut}
Let $A=F\left<x_1,\dots,x_n\right>/I$ be a projectively simple monomial algebra. Then every graded automorphism of $A$ is a composition of a torus action with a permutation on the generators. In particular, 
\[
\grAut(A)\leq \mathbb{G}_m \wr S_n=\mathbb{G}_m^{\times n}\rtimes S_n.
\]
\end{thm}

\begin{proof}
Let $w\in \{x_1,\dots,x_n\}^{\mathbb{N}}$ be a uniformly recurrent word such that $A=A_w$. Fix $\phi\in \grAut(A)$; we need to prove that there exist $\lambda_1,\dots,\lambda_n\in F^{\times}$ and $\sigma\in S_n$ such that for every $1\leq i\leq n$ we have $\phi(x_i)=\lambda_ix_{\sigma(i)}$. Extending scalars, we may assume without loss of generality that $F$ is an infinite field.
Assume to the contrary that for some $1\leq i\leq n$ we have $\phi(x_i)=\sum_{j=1}^{n} c_{ij}x_j$ with $c_{ij_1},c_{ij_2}\neq 0$ for some $j_1\neq j_2$. 
For each $1\leq j \leq n$, there exists some $1\leq k_j\leq n$ such that $\phi(x_{k_j})$ has $x_j$ in its support, for otherwise $x_j\notin \Im(\varphi)$ (recall that the presentation is assumed to be non-redundant). Consider the set: $$S=\{k_j|j\neq j_1,j_2\}\cup \{i\}$$ which has at most $n-1$ elements, say, $|S|=n-r$ for some $r\geq 1$, and enumerate $S=\{s_1,\dots,s_{n-r}\}$.
Consider the linear map $T\colon F^{n-r}\rightarrow A_1$ given by: $$T(\alpha_1,\dots,\alpha_{n-r})=\alpha_1\phi(x_{s_1})+\cdots+\alpha_{n-r}\phi(x_{s_{n-r}}).$$ For each $1\leq j\leq n$ let $\pi_j\colon A_1\rightarrow F$ be the linear functional projecting onto the coefficient of $x_j$. Notice that for each $j$, 
$$ \Ker (\pi_j \circ T) \subsetneq F^{n-r}. $$ Indeed, if $j\in\{j_1,j_2\}$ then $\phi(x_i)$ has $x_{j_1},x_{j_2}$ in its support, and so: $$\pi_j(T(0,\dots,0,1,0,\dots,0))=\pi_j(\phi(x_i))\neq 0$$ (where `$1$' is in the coordinate corresponding to $i$.) If $j\neq j_1,j_2$ then $\phi(x_{k_j})$ has $x_j$ in its support by definition and so $0\neq \pi_j(\phi(x_{k_j}))\in \Im(\pi_j \circ T)$. Since $F$ is an infinite field, an $F$-vector space cannot be equal to a union of finitely many of its proper subspaces, so: $$F^{n-r}\neq \bigcup_{j=1}^{n} \Ker(\pi_j\circ T).$$ Hence there exists a linear combination: $$\xi=\alpha_1x_{s_1}+\cdots+\alpha_{n-1}x_{s_{n-r}}$$ such that $\phi(\xi)=\beta_1x_1+\cdots+\beta_nx_n$ with all $\beta_1,\dots,\beta_n\neq 0$. We claim that $\beta_1x_1+\cdots+\beta_nx_n$ is a (right) regular element in $A$.
Indeed, it suffices to show that for every non-zero homogeneous element $\sum_l \gamma_{l=1}^p u_l$ (with $p\geq 1$ and $u_1,\dots,u_p$ non-zero monomials of the same length and $\gamma_1,\dots,\gamma_p\neq 0$) we have: $$\left(\sum_{l=1}^p \gamma_l u_l\right)\left(\beta_1x_1+\cdots+\beta_nx_n\right)\neq 0.$$ Since $u_1$ is a finite factor of $w$, it has some non-zero right prolongation, say, $u_1x_r\neq 0$. Then: $$\left(\sum_{l=1}^p \gamma_l u_l\right)\left(\beta_1x_1+\cdots+\beta_nx_n\right)=\sum_{l=1}^p \sum_{i=1}^n \gamma_l \beta_i u_l x_i$$ is non-zero since it is a linear combination of distinct monomials, one of which is $u_1x_r\neq 0$, whose coefficient is $\gamma_1\beta_r\neq 0$.
In particular, \[ \phi^{-1}(\beta_1x_1+\cdots+\beta_nx_n)=\xi=\alpha_1x_{s_1}+\cdots+\alpha_{n-1}x_{s_{n-r}} \] is regular too. 
Since $|S|<n$, we can pick some $t\in \{1,\dots,n\}\setminus S$.
But since $w$ is uniformly recurrent, there exists $C>0$ such that every length-$C$ factor of $w$ contains $x_t$. Therefore, there exists a maximal (finite) factor $v$ of $w$ which does not contain $x_t$; being a factor of $w$, observe that $v$ is a non-zero monomial in $A$. Now: \[ v(\alpha_1x_{s_1}+\cdots+\alpha_{n-1}x_{s_{n-r}})=\alpha_1vx_{s_1}+\cdots+\alpha_{n-1}vx_{s_{n-r}}=0 \] since each $vx_{s_1},\dots,vx_{s_{n-r}}$ is longer than $v$ and does not contain the letter $x_t$. This contradicts that $\alpha_1x_{s_1}+\cdots+\alpha_{n-1}x_{s_{n-r}}$ is regular. We conclude that there exists a function $\sigma\colon \{1,\dots,n\}\rightarrow \{1,\dots,n\}$ and scalars $\lambda_1,\dots,\lambda_n\in F$ such that $\phi(x_i)=\lambda_ix_{\sigma(i)}$ for all $1\leq i\leq n$. Since $\phi$ is an automorphism, $\sigma$ must be a permutation and $\lambda_1,\dots,\lambda_n\neq 0$, and the proof is completed.
\end{proof}

\begin{cor} \label{cor:iso}
If $f\colon A\rightarrow B$ is a graded isomorphism between projectively simple monomial algebras then $f$ is given by a permutation and scaling of the generators.
\end{cor}
\begin{proof}
By Theorem \ref{iso mon alg}, there exists a monomial isomorphism $g\colon B\rightarrow A$ (namely, $g$ permutes the generators). Then $g\circ f$ is a graded automorphism of $A$, so by Theorem \ref{Aut} it is a permutation and scaling of the generators of $A$, and therefore so is $f=g^{-1}\circ (g\circ f)$.
\end{proof}

\begin{rem}
The above proof applies more generally when $A$ takes the form $A_w$ for some infinite word $w$ and each generator $x_i$ generates a finite-codimensional two-sided ideal.
\end{rem}

\begin{exmpl} \label{Aut(TM)}
Let ${\bf t}$ be the Thue-Morse sequence discussed above in Subsection \ref{subsec:iso mon alg}:
\[ 01101001\cdots \]
This word can also be defined by $t_1=0,t_{n+1}=t_n\overline{w_n}$ where $\overline{u}$ is the bitwise negation of $u$, and ${\bf t}=\lim_{n \rightarrow \infty} t_n$. The Thue-Morse sequence is uniformly recurrent and a string of bits factors it if and only if its bitwise negation does. 
Let $A=F\left<x,y\right>/I$ be the monomial algebra associated with the Thue-Morse sequence, substituting $0\mapsto x$ and $1\mapsto y$. This is a projectively simple algebra, since ${\bf t}$ is uniformly recurrent, and the non-trivial permutation $\tau\colon x\mapsto y,\ y\mapsto x$ defines a graded automorphism of $A$. Therefore $\grAut(A)\cong \left(\mathbb{G}_m\times \mathbb{G}_m\right) \rtimes \left<\tau\right>$.
\end{exmpl}

\section{Proalgebraic parametrization spaces}

\subsection{Proalgebraic varieties} We recall the theory of proalgebraic varieties, which serves as the correct algebro-geometric setting for spaces of point modules of arbitrary noncommutative graded algebras. The notion goes back to Grothendieck \cite{Grothendieck} and in the symbolic dynamical context to Gromov \cite{Gromov,Gromov2}; see also \cite{KowPil,Yokura}.

\begin{defn} An $\mathbb{N}$-indexed \textit{proalgebraic variety} is a sequence $\{(V_n,\varphi_n)\}_{n=0}^{\infty}$ where each $V_n$ is an algebraic variety and each $\varphi_n\colon V_{n+1}\rightarrow V_n$ is a surjective morphism of algebraic varieties:
\[ \cdots \rightarrow V_2 \xrightarrow{\varphi_2} V_1 \xrightarrow{\varphi_1} V_0. \]
For $i\leq i'$, we define $\varphi^{i'}_i:=\varphi_{i'}\circ\cdots \circ \varphi_i$ (or $\text{id}_{V_i}$ if $i=i'$). 
Let $$\mathcal{V}=\{(V_n,\varphi_n)\}_{n=0}^{\infty}~{\rm and}~ \ \mathcal{W}=\{(W_n,\psi_n)\}_{n=0}^{\infty}$$ be proalgebraic varieties. A \textit{morphism} of ($\mathbb{N}$-indexed) proalgebraic varieties, or a \textit{proregular map} $f\colon \mathcal{V}\rightarrow \mathcal{W}$ consists of an order-preserving function\footnote{The function $g$ may be defined only for $n\gg 1$.} $g\colon \mathbb{N}\rightarrow \mathbb{N}$ and a family of morphisms of algebraic varieties $f_j\colon V_{g(j)}\rightarrow W_j$ such that if $j\leq j'$ then $f_j\circ \varphi^{g(j')}_{g(j)}=\psi^{j'}_j\circ f_{j'}$.

A \textit{point} of a proalgebraic variety $\mathcal{V}$ as above is a compatible sequence $(a_n)_{n=0}^ {\infty}$ with each $a_i\in V_i$ and $\varphi_i(a_{i+1})=a_i$. We denote the set of points of $\mathcal{V}$ by $\varprojlim \mathcal{V}$. We occasionally take the freedom to think of the set of points as our proalgebraic variety itself, where the sequence of varieties and morphisms along the sequence are obvious.
\end{defn}

\begin{exmpl}
Let $V_n=\left(\mathbb{P}^d\right)^{\times n}$ and $\varphi_n\colon V_{n+1}\rightarrow V_n$ the projection onto the first $n$ coordinates. Then $\mathcal{V}=\{(V_n,\varphi_n)\}_{n=1}^{\infty}$ is a proalgebraic variety. If $W_n=\left(\mathbb{P}^d\right)^{\times a_n}$ for some increasing sequence $(a_n)_{n=1}^{\infty}$ and $\psi_n\colon W_{n+1}\rightarrow W_n$ is the projection onto the first $a_n$-coordinates then formally $\mathcal{W}=\{(W_n,\psi_n)\}_{n=1}^{\infty}$ is different from $\mathcal{V}$, but they are naturally isomorphic to each other, and can both be identified with $\mathbb{P}^d\times \mathbb{P}^d\times~\cdots$.
\end{exmpl}

\subsection{Parametrizing point modules}

Let $A=F\left<x_1,\dots,x_m\right>/I$ be a graded algebra ($m>1)$. Point $A$-modules are parametrized by a proalgebraic variety \cite{Rogalski, SV}.

Let $P=Fe_0+Fe_1+\cdots$ be a point module and let $e_i\cdot x_j=\lambda_{i,j}e_{i+1}$ for some scalars $\lambda_{i,j}\in F$. Therefore, point modules are parametrized by a subset (in fact, a proalgebraic subvariety) of $\mathbb{A}^m\times \mathbb{A}^m\times \cdots$ by:
\begin{eqnarray*}
    \{\text{Point}\ A\text{-modules}\} & \rightarrow & \mathbb{A}^m\times \mathbb{A}^m\times \cdots \\
    P & \mapsto & (\underline{\lambda}_0,\underline{\lambda}_1,\dots)
\end{eqnarray*}
where $\underline{\lambda}_i=(\lambda_{i,1},\dots,\lambda_{i,m})\in \mathbb{A}^m$.
By Theorem \ref{Graded_isomorphisms}, isomorphism classes of point modules correspond to orbits by the action of $\mathbb{G}_m\times \mathbb{G}_m\times\cdots$ so:
\begin{eqnarray} \label{proalg}
    \{\text{Point}\ A\text{-modules}\}/\cong & \rightarrow & \mathbb{P}^{m-1}\times \mathbb{P}^{m-1}\times \cdots \\
    P & \mapsto & (\underline{\lambda}_0,\underline{\lambda}_1,\dots)
\end{eqnarray}
where, from now on, $\underline{\lambda}_i=[\lambda_{i,1}\colon\cdots\colon\lambda_{i,m}]\in \mathbb{P}^{m-1}$.
The subset of $\mathbb{P}^{m-1}\times \mathbb{P}^{m-1}\times \cdots$ parametrizing point modules can be naturally identified with the set of points of the following proalgebraic variety $\mathcal{P}(A)$.
Let $\mathcal{P}_n(A)\subseteq \left(\mathbb{P}^{m-1}\right)^{\times (n+1)}$ be the projection of the above set onto the first $n+1$ copies of $\mathbb{P}^{m-1}$; equivalently, $\mathcal{P}_n(A)$ parametrizes truncations of point modules of $A$ with maximum non-zero degree $n$. We let $\pi_n\colon \mathcal{P}_{n+1}(A)\twoheadrightarrow \mathcal{P}_n(A)$ be the restriction of the projection $\left(\mathbb{P}^{m-1}\right)^{\times (n+2)} \twoheadrightarrow \left(\mathbb{P}^{m-1}\right)^{\times (n+1)}$. Then: \[ \mathcal{P}(A)=\{(\mathcal{P}_n(A),\pi_n)\}_{n=0}^{\infty} \] is a proalgebraic variety whose set of points naturally corresponds\footnote{Formally, $\varprojlim \mathcal{P}(A)$ is a subset of $\mathbb{P}^{m-1}\times \left(\mathbb{P}^{m-1}\right)^{\times 2}\times \left(\mathbb{P}^{m-1}\right)^{\times 3}\times \cdots$ but we naturally identify it with $\mathbb{P}^{m-1}\times \mathbb{P}^{m-1}\times \mathbb{P}^{m-1}\times \cdots$.} to the image of the parametrization given in \ref{proalg}. By abuse of notation, let us also think of $\pi_n$ as the natural projection from $\varprojlim \mathcal{P}(A)$ or from $\mathcal{P}_{n'}(A)$  onto $\mathcal{P}_n(A)$, whenever $n'\geq n$.

\begin{rem}
\begin{enumerate}
    \item The set of points parametrizing point modules $\varprojlim \mathcal{P}(A)$ forms a dynamical system with respect to the semigroup $\mathbb{N}=\left<s\right>$ acting by shifting $s\colon (\underline{\lambda}_0,\underline{\lambda}_1,\dots)\mapsto (\underline{\lambda}_1,\underline{\lambda}_2,\dots)$. Indeed, if $P=Fe_0+Fe_1+\cdots$ is a point $A$-module then its shift $s(P):=Fe_1+Fe_2+\cdots$ is again a point $A$-module.
    
    \item Rogalski and Zhang \cite{RZ}, Smith \cite{Smith_Skl}, Walton \cite{Walton_Skl,Walton_Skl_cor} and others studied cyclic graded modules with Hilbert series $1+t+\cdots+t^n$, which can be thought of as `truncated point modules'. But one must be careful: while a point module can be truncated by $M\mapsto M/M_{\geq n+1}$ (for each natural number $n$), in general, truncated point modules as above need not arise as truncations of actual point modules. Let $\overline{\mathcal{P}}_n(A)$ denote the moduli space of cyclic graded modules with Hilbert series $1+t+\cdots+t^n$.    
    Thus, in general, 
    \[ \mathcal{P}_n(A)\subseteq \overline{\mathcal{P}}_n(A)\]
    is a proper inclusion (see Example \ref{Exmpl:trunc}). In the example studied in \cite{Smith} we have $\mathcal{P}_n(A)=\overline{\mathcal{P}}_n(A)$ for all $n$, and in the example studies in \cite{Smith_Skl,Walton_Skl,Walton_Skl_cor}, $\mathcal{P}_n(A)=\overline{\mathcal{P}}_n(A)$ for all $n\geq 2$.    
    However, in all cases, $\varprojlim \mathcal{P}_n(A)=\varprojlim \overline{\mathcal{P}}_n(A)$.
\end{enumerate}

\begin{exmpl} \label{Exmpl:trunc}
Let $A=F\left<x_1,x_2,x_3,x_4\right>/\left<x_1x_2,x_1x_4,x_2x_1,x_2x_3\right>$, a prime monomial algebra. For each $n\geq 2$, the point module:
\[
M = Fe_0+\cdots+Fe_n
\]
given by the action $e_i\cdot x_j = \delta_{j,4}e_{i+1}$ for $1\leq i\leq n-2$ and by $e_{n-1}\cdot x_j = e_n$ for all $1\leq j\leq 4$ is not a truncation of a point module. Indeed, if there was a point module $\widetilde{M}=Fe_0+Fe_1+\cdots$ such that $\widetilde{M}/\widetilde{M}_{\geq n+1}\cong M$ then for some $1\leq j\leq 4$ we had $e_n \cdot x_j = \lambda e_{n+1}$ for some $\lambda\neq 0$. If $j\in \{1,3\}$ then $e_{n+1} = e_{n-1}\cdot \lambda^{-1} x_2x_j = 0$, and if $j\in \{2,4\}$ then $e_{n+1} = e_{n-1}\cdot \lambda^{-1} x_1x_j = 0$, a contradction.
\end{exmpl}

\end{rem}
To what extent does $\mathcal{P}(A)$ depend on the presentation of $A$? By Theorem \ref{iso mon alg}, if we fix two non-redundant monomial presentation of the same algebra then each one is induced from the other by a permutation on the generators. In other words, the resulting proalgebraic varieties lie in the same orbit of $S_m$ under its diagonal action on $\mathbb{P}^{m-1}\times \mathbb{P}^{m-1}\times \cdots$.

Our motivating question from now on is: To what extent is a given graded algebra recoverable from its proalgebraic variety of point modules?
Since $\mathcal{P}(A)$ could be empty---when $A$ admits no point modules---one cannot hope to reconstruct $A$ from $\mathcal{P}(A)$ in general.
Even for monomial algebras, in which case $\mathcal{P}(A)$ is never empty, we have $\mathcal{P}(A)\cong \mathcal{P}(A/\sing(A))$ by Lemma \ref{lem_prolongable} so one can potentially recover $A$ only up to an extension by a prolongable radical; equivalently, only within prolongable monomial algebras (in view of Corollary \ref{prolongable cor}). 

\subsection{Twisting points}
Let $A,B$ be two connected graded algebras, finitely generated in degree $1$. Zhang \cite{ZhangTwists} gave a remarkable characterization of when their categories of graded right modules with degree preserving morphisms are equivalent $\text{Gr}(A)\equiv \text{Gr}(B)$. Namely, he proved that these categories are equivalent if and only if $B$ is isomorphic to a \textit{Zhang twist} of $A$. A twisting system $\tau$ of $A$ consists of a sequence of degree-preserving linear automorphisms $\tau = \{\tau_n\}_{n=0}^{\infty}$ satisfying:
\[ \tau_n(x\tau_m(y)) = \tau_n(x)\tau_{n+m}(y) \] for every $n,m\geq 0$ and for every homogeneous elements $x,y\in A$ such that $x\in A_m$.
Then the twist of $A$ with respect to $\tau$ is the new graded algebra $A^\tau$ in which multiplication is given by (linearly extending) $x\star y:=x\tau_n(y)$ where $x,y$ are homogeneous and $n$ is the degree of $x$. A special case of a twisting system is given by $\{\tau_n=f^{\circ n}\}_{n=0}^{\infty}$ where $f\colon A\rightarrow A$ is a degree-preserving automorphism.

One obtains a bijective functorial correspondence of graded (right) $A$-modules with graded $A^\tau$-modules via twisting: if $M=\bigoplus_{i=0}^{\infty} M_i$ is a graded $A$-module then $M^\tau=\bigoplus_{i=0}^{\infty} M_i$ and $v\star a:=v\cdot \tau_n(a)$ where $v\in M_n$ and $a\in A^\tau$ a homogeneous element.

\begin{lem} \label{Zhang_prolongable}
Let $A$ be a connected graded algebra, finitely generated in degree $1$, and let $\tau$ be a twisting system of $A$. Then $A$ is right prolongable if and only if $B$ is, and $A$ is left prolongable if and only if $B$ is.
\end{lem}
\begin{proof}
By symmetry, it suffices to prove that $\sing_r(A)\neq 0\implies \sing_r(B)$ and that $\sing_l(A)\neq 0\implies \sing_l(B)$.

Assume that $A$ is not right prolongable, and pick a non-zero element $a\in \sing_r(A)\cap A_d$ for some $d$ such that $aA_{\geq 1}=0$ (recall that the prolongable radical is homogeneous). Now for every $b\in B_{\geq 1}$ we have:
\[
a \star b = a \cdot \tau_d(b) = 0
\]
so viewing $a$ as an element of $B$ we have that $a\in \sing_r(B)\neq 0$.

Now assume that $A$ is not left prolongable, and pick $a\in \sing_l(A)\cap A_d$ for some $d$ such that $A_{\geq 1}a=0$. Let $b=\tau_1^{-1}(a)$. For every $c\in B_1$ we have that: 
\[
c\star b=c\tau_1(b)=ca=0,
\]
and since $B=A^\tau$ is also generated in degree $1$, it follows that $B_{\geq 1}\star b=0$. Hence $b\in \sing_l(B)\neq 0$.
\end{proof}

\begin{prop} \label{ZhangTwistProvarieties}
Let $A$ be a connected graded algebra and let $A^\tau$ be its twist with respect to some twisting system. Then $\mathcal{P}(A)\cong \mathcal{P}(A^\tau)$ and moreover, for each $n$ there is an isomorphism of algebraic varieties $f_n\colon \mathcal{P}_n(A)\rightarrow \mathcal{P}_n(A^\tau)$ such that $\pi_n\circ f_n = f_{n-1}\circ \pi_n$.
\end{prop}

\begin{proof}
Let $A=F\left<x_1,\dots,x_m\right>/I$ be a connected graded algebra and fix a twisting system over it $\tau=\{\tau_n\}_{n=0}^{\infty}$. For each $i\geq 0$ and $1\leq j\leq m$ write $\tau_i(x_j)=\sum_{k=1}^{m} c_{ij}^{(k)}x_k$ for some $c_{ij}^{(k)}\in F$.

The assignment $M\mapsto M^{\tau}$ is a bijection between $\varprojlim \mathcal{P}(A)$ and $\varprojlim \mathcal{P}(A^{\tau})$. It remains to check that the induced maps on the truncations $f_n\colon \mathcal{P}_n(A)\rightarrow \mathcal{P}_n(A^\tau)$ satisfy the proposition requirements. Fix 
$(\underline{\lambda}_0,\underline{\lambda}_1,\dots)\in \varprojlim \mathcal{P}(A)$
where $\underline{\lambda}_i=(\lambda_{i,1},\dots,\lambda_{i,m})\in \mathbb{P}^{m-1}$ and $M=Fe_0+Fe_1+\cdots$ is a point module with $e_i\cdot x_j=\lambda_{i,j}e_{i+1}$. Now the point module $M^{\tau}$ is parametrized by $(\underline{\lambda}'_0,\underline{\lambda}'_1,\dots)$
where $\underline{\lambda}'_i=(\lambda'_{i,1},\dots,\lambda'_{i,m})$ and $\lambda'_{i,j}=\sum_{k=1}^{m} c_{ij}^{(k)} \lambda_{i,k}$ since: \[ e_i\star x_j=e_i \cdot \tau_i(x_j)=e_i\cdot \left(\sum_{k=1}^{m} c_{ij}^{(k)}x_k\right)=\left(\sum_{k=1}^{m} c_{ij}^{(k)} \lambda_{i,k}\right)e_{i+1}. \]
Therefore twisting induces an isomorphism of proalgebraic varieties:
\[ \begin{tikzcd}
\left(\mathbb{P}^{m-1}\right)^{\times n+1} \arrow[rr, "g_0\times \cdots \times g_n"] &  & \left(\mathbb{P}^{m-1}\right)^{\times n+1} \\
\mathcal{P}_n(A) \arrow[rr, "f_n"] \arrow[u, hook]                                   &  & \mathcal{P}_n(A^\tau) \arrow[u, hook]     
\end{tikzcd} \]
where the vertical maps are inclusions and $g_i\colon \mathbb{P}^{m-1}\rightarrow \mathbb{P}^{m-1}$ is a projective general linear automorphism given by the matrix:
\[ \left(\begin{matrix} c^{(1)}_{i1} & \cdots & c^{(m)}_{i1} \\ \vdots & \ddots & \vdots \\ c^{(1)}_{im} & \cdots & c^{(m)}_{im} \end{matrix}\right)\]
and therefore $f_{n-1}\circ \pi_n = \pi_n \circ f_n$.
\end{proof}

\begin{exmpl}
Let ${\bf t}$ be the Thue-Morse sequence and consider it under the substitution $0\mapsto x$ and $1\mapsto y$:
\begin{align}
   xyyxyxxy\cdots \label{TM_substitute}
\end{align}
Let $A=F\left<x,y\right>/I$ be the monomial algebra associated with it.
Then $A$ is a projectively simple algebra of GK-dimension 2. Consider the automorphism $\tau\colon A\rightarrow A$ given by $\tau(x)=y,\tau(y)=x$ from Example \ref{Aut(TM)}. Let $B$ be the twist of $A$ by $\tau$. 
By Proposition \ref{ZhangTwistProvarieties}, $\mathcal{P}(A)\cong \mathcal{P}(B)$. However, $A\not\cong B$, for otherwise there would be a degree preserving isomorphism \cite{BellZhang} $f\colon B\rightarrow A$. However, in $B$ we have that:
\[
y^{\star 4}=y\tau(y)\tau^2(y)\tau^3(y)=yxyx\neq 0
\]
as seen by (\ref{TM_substitute}). Furthermore, the Thue-Morse sequence is cube-free \cite{MH}, namely, for every monomial $w\in A$ we have $w^3=0$. Therefore:
\[
y^{\star 6}=y\tau(y)\tau^2(y)\tau^3(y)\tau^4(y)\tau^5(y)=yxyxyx = 0
\]
so $f(y)$ cannot be of the form $\lambda x$ or $\lambda y$, but if $\lambda,\mu\neq 0$ then $\lambda x + \mu y$ is not nilpotent, since its $n$-th power is a linear combination supported on all non-zero length-$n$ monomials in $A$, a contradiction.
\end{exmpl}

\subsection{Proalgebraic varieties and trees}

Fix a finite alphabet $\Sigma=\{x_1,\dots,x_m\}$ with $m\geq 2$. Let: \[ p_0=[1\colon 0\cdots \colon 0],\dots,p_{m-1}=[0\colon \cdots \colon 0\colon 1]\in \mathbb{P}^{m-1}. \]
For a non-empty subset $C\subseteq \Sigma$, let:
\[
i(C) = \{[w_0\colon\cdots \colon w_{m-1}]\in \mathbb{P}^{m-1}\ |\ \forall k\in C,\ w_{k-1}=0 \}
\]
so, in particular, $i(\{x_k\})=\{p_{k-1}\}$.



\begin{thm} \label{decomposition}
Let $F\left<x_0,\dots,x_m\right>/I$ be a monomial algebra. Then:
\[ 
\varprojlim \mathcal{P}(A) = \bigcup_{\substack{\mathcal{T}=C_0C_1\cdots \\ \text{tree over}\ A}} i(C_0)\times i(C_1)\times \cdots 
\]
Equivalently,
\[ \mathcal{P}_n(A) = \bigcup_{\substack{\mathcal{T}=C_0C_1\cdots \\ \text{tree over}\ A}} i(C_0)\times i(C_1)\times \cdots \times i(C_n). \]
Similarly, $\overline{\mathcal{P}}_n(A)$ is a union over:
\[
\{ i(C_0)\times\cdots\times i(C_n)\ |\ \text{All of the monomials in}\ C_0\cdots C_n\ \text{are non-zero in}\ A \}.
\]
\end{thm}

\begin{proof}
This is immediate from Theorem \ref{Classification}.
\end{proof}

Therefore, the proalgebraic variety of point modules of a monomial algebra is a projective limit of projective varieties, each of which is a finite union of products of projective spaces and points. This proalgebraic variety can be seen as a `linearization' of the underlying subshift of the maximal prolongable quotient of the given algebra.

\begin{exmpl} \label{exmpl:free_proj}
Let $F\left<\Sigma\right>$ where $\Sigma=\{x_1,\dots,x_d\}$ be a $d$-generated free algebra for $d\geq 2$. Then $\Sigma\Sigma\cdots$ is a tree over $F\left<\Sigma\right>$ and so $\mathcal{P}_n(F\left<\Sigma\right>)=\left(\mathbb{P}^{d-1}\right)^{\times(n+1)}$ and $\varprojlim \mathcal{P}_n(F\left<\Sigma\right>)=\mathbb{P}^{d-1}\times \mathbb{P}^{d-1}\times\cdots$.
\end{exmpl}

\begin{exmpl}
Let $A=F\left<x,y\right>/\left<y\right>^2$. A tree $C_0C_1\cdots$ over $A$ is given by $C_i=\{x\}$ for all $i$ except for at most one index $i=i_0$, for which $C_{i_0}$ is allowed to be $\{x,y\}$. Consequently, $\varprojlim\mathcal{P}(A)=\mathbb{P}^1\cup \mathbb{P}^1\cup \cdots$ is a countably infinite union of projective lines with a common intersection point.   
\end{exmpl}



\begin{exmpl}
Let $A=F\left<x,y\right>/\left<xy\right>$. A tree $C_0C_1\cdots$ over $A$ is given by either:
\begin{itemize}
    \item $C_i=\{y\}$ for all $i\geq 0$;
    \item $C_i=\{y\}$ for all $i<d$, then $C_{d}=\{x,y\}$ or $C_d=\{x\}$ and $C_i=\{x\}$ for all $i\geq d+1$ (possibly $d=0$);
    \item $C_i=\{x\}$ for all $i\geq 0$.
\end{itemize}
Therefore $\mathcal{P}_n(A)$ is a union of $n+1$ projective lines $L_0,\dots,L_n\cong \mathbb{P}^1$ where: 
\[ 
L_i\cap L_j= \begin{cases} 
      \{\bullet\} & \text{if}\ |i-j|=1 \\
      \emptyset & \text{if}\ |i-j|>1.
   \end{cases}
\]
\end{exmpl}

\section{Irreducibility and rigidity} \label{sec:irr free}


In this section we work over an algebraically closed field $F$.
A proalgebraic variety $\{(V_n,\varphi_n)\}_{n=1}^{\infty}$ is \emph{irreducible} if each $V_n$ is an irreducible algebraic variety.
As seen in Example \ref{exmpl:free_proj}, the proalgebraic variety of point modules of a free algebra $F\left<x_1,\dots,x_m\right>$ is an infinite product of $\mathbb{P}^{m-1}$'s, and is thus irreducible. There exist non-free prolongable monomial algebras with an irreducible proalgebraic variety of point modules; for instance, for the (right) prolongable monomial algebra $A=F\left<x_1,x_2,x_3\right>/\left<x_1^2,x_2x_1,x_3x_1\right>$ we have that $\mathcal{P}_n(A)=\mathbb{P}^2\times \left(\mathbb{P}^1\right)^{\times n-1}$, which is irreducible. However, algebras with irreducible varieties are close to being free: they are all `nilpotent by free'.
\begin{prop} \label{Irr}
Let $A$ be a prolongable monomial algebra. If $\mathcal{P}(A)$ is irreducible then there exists a nilpotent ideal $N\triangleleft A$ generated by a (possibly empty) subset of the set of monomial generators of $A$, such that $A/N\cong F\left<x_1,\dots,x_k\right>$ for some $k$.
\end{prop}
\begin{proof}
Let $\Sigma=\{x_1,\dots,x_m\}$ and let $A=F\left<\Sigma\right>/I$ be a prolongable monomial algebra with an irreducible $\mathcal{P}(A)$. By Theorem \ref{decomposition},
\[ 
\mathcal{P}_n(A) = \bigcup_{\substack{\mathcal{T}=C_0C_1\cdots \\ \text{tree over}\ A}} i(C_0)\times i(C_1)\times \cdots \times i(C_n). 
\]
By assumption, $\mathcal{P}_n(A)$ is irreducible and hence, since it is a union of finitely many subsets, one of them must contain all of the others. This means that there exists a tree over $A$, say, $\mathcal{T}_{\text{max}}^{(n)}=C_0C_1\cdots$ such that for every other tree $\mathcal{T}=C'_0C'_1\cdots$ over $A$ it holds that $C'_i\subseteq C_i$ for all $0\leq i\leq n$; we say that $\mathcal{T}_{\text{max}}^{(n)}$ is an \textit{$n$-maximal} tree. If $\mathcal{T}$ is $n'$-maximal for some $n'\geq n$ then it is also $n$-maximal. It follows by K\"onig's lemma 
that there exists a tree $\mathcal{T}=C_0C_1\cdots$ which is $n$-maximal for all $n\geq 0$.
Now $\mathcal{T}'=C_1C_2\cdots$ is also a tree over $A$ and therefore by $n$-maximality of $\mathcal{T}$ we get that the $i$-th set of $\mathcal{T}'$, which is $C_{i+1}$, is contained in the $i$-th set of $\mathcal{T}$, which is $C_i$. Hence $C_0\supseteq C_1\supseteq \cdots$. In particular, there exists some $d\geq 0$ such that $C_d=C_{d+1}=\cdots$. Let $k=|C_d|$. 
Since $A$ is prolongable, every non-zero monomial in it factors some infinite word from some tree over $A$, and therefore factors some infinite word from the tree $\mathcal{T}$; hence $A=A_\mathcal{T}$.
Since every monomial in the free algebra $F\left<C_d\right>$ factors some infinite word from $C_dC_{d+1}\cdots$, it follows that $A/\left<\Sigma\setminus C_d\right>\cong F\left<C_d\right>\cong F\left<x_1,\dots,x_k\right>$. Moreover, the ideal $N=\left<\Sigma\setminus C_d\right>\triangleleft A=A_{\mathcal{T}}$ is nilpotent, as $N^{d+1}=0$.
\end{proof}

\begin{rem} \label{ShouldBeProp}
The proof of Proposition \ref{Irr} gives us the following: if $\mathcal{P}(A)$ is irreducible then $\mathcal{P}_n(A)=i(C_0)\times \cdots \times i(C_n)$ where $\mathcal{T}=C_0C_1\cdots$ is an $n$-maximal tree over $A$ for all $n$ (such a tree exists, as shown in the proof) and $A=A_{\mathcal{T}}$.
Furthermore, if $\dim i(C_n)=k$ for all $n\geq 0$ then $C_0=C_1=\cdots$ and $A\cong F\left<x_1,\dots,x_k\right>$.
\end{rem}

Notice that by Proposition \ref{Irr}, there are only countably many monomial algebras whose proalgebraic variety of point modules is irreducible.

\begin{thm} \label{free_rigid}
If $A$ is a prolongable monomial algebra with the property that $\mathcal{P}(A)\cong \mathcal{P}(F\left<x_1,\dots,x_d\right>)$, then $A\cong F\left<x_1,\dots,x_d\right>$.
\end{thm}

\begin{proof}
Fix morphisms: $$ \{f_n\colon \mathcal{P}_{g(n)}(F\left<x_1,\dots,x_d\right>)=(\mathbb{P}^{d-1})^{\times \left(g(n)+1\right)}\rightarrow \mathcal{P}_n(A)\} $$ and: $$ \{f'_n\colon \mathcal{P}_{g'(n)}(A)\rightarrow \mathcal{P}_n(F\left<x_1,\dots,x_d\right>)=(\mathbb{P}^{d-1})^{\times (n+1)}\} $$ such that the induced maps between $\varprojlim \mathcal{P}(F\left<x_1,\dots,x_d\right>)$ and $\varprojlim \mathcal{P}(A)$ are mutual inverses.
For each $n$, 
\[
\mathcal{P}_n(A) = \bigcup_{m=n}^{\infty} \Im\left(\pi^m_n \circ f_m\right)
\]
(where $\pi^m_n\colon \mathcal{P}_m(A)\rightarrow \mathcal{P}_n(A)$ is the natural projection.) Since each $\mathcal{P}_m(A)$ is an irreducible projective variety, each $\Im\left(\pi^m_n \circ f_m\right)\subseteq \mathcal{P}_n(A)$ is an irreducible projective variety and hence the union stabilizes, and $\mathcal{P}_n(A)=\Im\left(\pi^{m_0}_n \circ f_{m_0}\right)$ for some $m_0\geq n$ is irreducible itself and consequently, $\mathcal{P}(A)$ is irreducible.
It now follows from Proposition \ref{Irr} that there exists some $r\geq 0$ and $a_0 \geq \cdots\geq a_r$ such that for every $n$,
\[ 
\mathcal{P}_n(A)= i(C_0)\times \cdots \times i(C_n) 
\] 
where: 
\[
\dim i(C_0)=a_0,\dots, \dim i(C_r)=a_r,\dots, \dim i(C_n)=a_r.
\] 
By Remark \ref{ShouldBeProp}, it suffices to show that $a_0=\cdots=a_r=d-1$, from which it follows that $A\cong F\left<x_1,\dots,x_d\right>$.

If $a_0>d-1$ then for each $n$, $\mathcal{P}_{n}(A)=\mathbb{P}^{a_0}\times i(C_1)\times \cdots \times i(C_n)$ and for each $p\in i(C_1)\times \cdots \times i(C_{g'(n)})$ we have a regular map:
\begin{align}
f'_n(-,p)\colon \mathbb{P}^{a_0}\rightarrow (\mathbb{P}^{d-1})^{\times (n+1)}. \label{rigidity1}
\end{align}
Let $\text{pr}_0,\dots,\text{pr}_{n+1}\colon \left(\mathbb{P}^{d-1}\right)^{\times (n+1)}\rightarrow \mathbb{P}^{d-1}$ be the standard projections. Then each $\text{pr}_i \circ f'_n(-,p)\colon \mathbb{P}^{a_0}\rightarrow \mathbb{P}^{d-1}$ is constant since $a_0>d-1$, and hence so must be $f'_n(-,p)$ itself. (Recall that a regular map $\mathbb{P}^{n_1}\rightarrow \mathbb{P}^{n_2}$ with $n_1>n_2$ must be constant, since every homogeneous ideal in $F[t_0,\dots,t_{n_1}]$ generated by $n_2+1$ homogeneous polynomials vanishes at some point, so the map is not defined at this point.) 
Since for every $n$ the map $f'_n(-,p)$ as above (\ref{rigidity1}) is constant for each fixed $p$, it follows that for every $\tilde{p}\in i(C_1)\times i(C_2)\times \cdots$ the points $(0,\tilde{p}),(1,\tilde{p})\in \varprojlim \mathcal{P}(A)$ are mapped to the same point under $\varprojlim f'_n$, a contradiction. Hence $a_0\leq d-1$.

It remains to show that $a_r\geq d-1$; otherwise, assume that $a_r<d-1$. We may readily assume that $d\geq 2$. We claim that for each $r\leq k\leq n$ the natural projection:
\begin{align}
\text{pr}_{i(C_k)} \circ f_{n}\colon (\mathbb{P}^{d-1})^{\times (g(n)+1)}\rightarrow 
i(C_k) \cong \mathbb{P
}^{a_r} \label{rigidity2}
\end{align}
is constant. If $a_r=0$, this is clear; assume $a_r>0$ and let $l=g(n)+1$. 
If $l=1$ we get a regular map $\mathbb{P}^{d-1}\rightarrow \mathbb{P}^{a_0}$ which is thus constant, since $a_0<d-1$. 
For $l\geq 2$, it follows by induction that the map $\text{pr}_{i(C_k)} \circ f_{n}$ is constant on the fibers:
\[
(\mathbb{P}^{d-1})^{\times (l-1)}\times \{w\}\ \text{and}\  \{v\}\times \mathbb{P}^{d-1}
\]
for every $w\in \mathbb{P}^{d-1}$ and $v\in (\mathbb{P}^{d-1})^{\times (l-1)}$. By the rigidity theorem (see e.g.  \cite[Theorem~2.1]{MilneRigidity}), such a map must be constant. Thus $\text{pr}_{i(C_k)}\circ f_{n}$ is constant for each $r\leq k\leq n$. It follows that: 
\[
\dim \Im (f_n)\leq a_0+\cdots+a_{r-1}=O_n(1).
\]
Now take $n=g'(m)$ for some $m\gg 1$. Then $f'_{m}\circ f_{n}=\pi^{g(n)}_m$ must have image $(\mathbb{P}^{d-1})^{m+1}$ of dimension $(d-1)(m+1)$, while $\Im(f_n)$ has a bounded dimension, leading to a contradiction. Hence, $a_0=\cdots =a_r=d-1$, completing the proof.
\end{proof}



\section{Finitely presented monomial algebras}

In this section we are interested in the level of complexity of the spaces of point modules and of truncated point modules of monomial algebras,  and in particular, in their irreducible components. These have attracted considerable attention, partially thanks to their occurrence as degenerations of noncommutative surfaces (\cite{Smith_Skl,Walton_Skl,Walton_Skl_cor}; see also \cite{Smith}). 

\begin{prop} \label{prop:pro fin prsntd}
Let $A=F\left<x_1,\dots,x_m\right>/\left<w_1,\dots,w_r\right>$ be a finitely presented monomial algebra. Then $\sing(A)$ is finitely generated as a left ideal.
\end{prop}

\begin{proof}
Recall that by Proposition \ref{sing is mon}, $\sing(A)$ is generated by monomials. Let $u\neq 0$ be a monomial of length $e\geq d$ (we retain the notation of $d$ as a bound on the lengths of all of the monomial relations $w_1,\dots,w_r$). Decompose $u=u_0u_1$ where $|u_0|=e-d$ (possibly, $u_0$ is the empty monomial) and $|u_1|=d$. Then $u\in \sing(A)$ if and only if $u_1\in \sing(A)$. Indeed, the `if' part is obvious since $\sing(A)$ is an ideal; for the `only if' part, assume that $u\cdot A_{\geq r}=0$ for some $r\geq 1$. Then for every monomial $p\in A_{\geq r}$ there is a relation $w_j$ which factors $up=u_0u_1p$. Since $w_j$ cannot factor $u_0u_1$ and $|u_1|=d\geq |w_j|$, it follows that $w_j$ factors $u_1p$, which is thus equal to zero in $A$; hence $u_1\in \sing(A)$. 
Therefore $\sing(A)\cap A_{\geq d} = A \cdot \left(\sing(A)\cap A_d\right)$. Thus $\sing(A)$ is generated as a left ideal by the set $\sing(A)\cap A^{\leq d}$, which is finite.
\end{proof}

Therefore, if $A$ is a finitely presented monomial algebra then so is its maximal prolongable quotient $A/\sing(A)$.

For a connected graded algebra $A$ let: 
\begin{eqnarray*}
a_n(A) & = & \#\{\text{Irreducible components of}\ \mathcal{P}_n(A)\}; \\
\bar{a}_n(A) & = & \#\{\text{Irreducible components of}\ \overline{\mathcal{P}}_n(A)\}.
\end{eqnarray*}

\begin{thm} \label{rational}
Let $A$ be a finitely presented monomial algebra.
Then the generating functions: $$ \sum_{n=0}^{\infty} a_n(A)t^n\ \text{and}\ \sum_{n=0}^{\infty} \bar{a}_n(A)t^n $$ are rational functions.
\end{thm}


By Ufnarovskii's work, it is known that the Hilbert series $$H_A(t)=\sum_{n=0}^{\infty} \left(\dim_F A_n\right)t^n$$ of a finitely presented monomial algebra $A$ is rational (\cite{Uf1,Uf2}; see also \cite{BBL}).

For the rest of the section, fix a finitely presented monomial algebra $A=F\left<x_1,\dots,x_m\right>/\left<w_1,\dots,w_r\right>$ and let $a_n=a_n(A),\bar{a}_n=\bar{a}_n(A)$. By Proposition \ref{prop:pro fin prsntd}, we may assume that $A$ is prolongable. Let $d\geq 2$ be a bound on the lengths of all defining relations, that is, $\deg(w_1),\dots,\deg(w_r)\leq d$.
Let $\Sigma=\{x_1,\dots,x_m\}$ and let $\mathcal{S}=P(\Sigma)\setminus \{\emptyset\}$. 
A sequence $(C_1,\dots,C_k)\in \mathcal{S}^k$ is \textit{coherent} if all words in $C_1\cdots C_k$ are non-zero in $A$. 
Notice that a monomial in $\Sigma$ is non-zero in $A$ if and only if all of its length-$d$ factors are non-zero, and therefore a sequence of subsets of $\Sigma$ is coherent if and only if all of its length-$d$ sub-sequences are coherent.

A sequence $(C_1,\dots,C_k)\in \mathcal{S}^k$ is \textit{prolongable} if there exist sets $C_{k+1},C_{k+2},\dots\in \mathcal{S}$ such that $C_1C_2\cdots$ is a tree over $A$. Every prolongable sequence is coherent. A coherent sequence is prolongable if and only if its length-$(d-1)$ suffix is a prolongable sequence, since if $(C_1,\dots,C_{k+d-1})$ is coherent and $C_{k+1}C_{k+2}\cdots$ is a tree then $C_1C_2\cdots$ is a tree, since all of its length-$d$ sub-sequences are coherent.

We say that $(C_1,\dots,C_k)\preceq (C'_1,\dots,C'_k)$ if $C_i\subseteq C'_i$ for all $1\leq i\leq k$. A coherent (resp. prolongable) sequence is \textit{coherent-maximal} (resp. \textit{prolongable-maximal}) if it is $\preceq$-maximal among all coherent (resp. prolongable) sequences of the same length.

By Theorem \ref{decomposition},
\begin{eqnarray*}
\bar{a}_n & = & \#\{\text{Coherent-maximal sequences}\ (C_0,\dots,C_n)\},\ \text{and} \\
a_n & = & \#\{\text{Prolongable-maximal sequences}\ (C_0,\dots,C_n)\}.
\end{eqnarray*}

A sequence $(C_1,\dots,C_{k+d-1})$ is \textit{coherent-premaximal} (respectively \textit{prolongable-premaximal}) if it is coherent (respectively prolongable) and $(C_1,\dots,C_k)$ is maximal among sequences $(X_1,\dots,X_k)$ such that $(X_1,\dots,X_k,C_{k+1},\dots,C_{k+d-1})$ is coherent (respectively prolongable).

A sequence $(C_1,\dots,C_{k+d-1})$ for $k\geq d-1$ is \textit{coherent-postmaximal} (respectively \textit{prolongable-postmaximal}) if it is coherent (respectively prolongable) and $(C_{k+1},\dots,C_{k+d-1})$ is maximal among all of the sequences $(X_1,\dots,X_k)$ such that $(C_{k-d+2},\dots,C_k,X_1,\dots,X_{d-1})$ is coherent (respectively prolongable).

\begin{lem} \label{lem:maxtopremax}
Let $n\geq d-1$. A 
sequence $\vec{C}=(C_0,\dots,C_n)$ is coherent-maximal (respectively prolongable-maximal) if and only if:
\begin{enumerate}
    \item $\vec{C}$ is coherent-premaximal (respectively prolongable-premaximal), and
    \item $\vec{C}$ is coherent-postmaximal (respectively prolongable-postmaximal).
\end{enumerate}
\end{lem}

\begin{proof}
For the `if' direction, suppose that the sequence $\vec{C}=(C_0,\dots,C_n)$ is a coherent-premaximal and coherent-postmaximal (resp. prolongable-premaximal and prolongable-postmaximal). Then by definition, $\vec{C}$ is coherent (respectively prolongable); if it is not coherent-maximal (respectively prolongable-maximal) then there exists some $0\leq i\leq n$ and a strict inclusion $C_i\subset C'_i$ such that $\vec{C}':=(C_0,\dots,C'_i,\dots,C_n)$ is coherent (respectively prolongable). If $0\leq i\leq n-d+1$ then $\vec{C}\prec \vec{C}'$ is not coherent-premaximal (respectively prolongable-premaximal); and if $n-d+2\leq i\leq n$ then $\vec{C}\prec \vec{C}'$ is not coherent-postmaximal (respectively prolongable-postmaximal).

For the `only if' direction, assume that $\vec{C}=(C_0,\dots,C_n)$ is coherent-maximal (respectively prolongable-maximal). If $\vec{C}$ is not coherent-premaximal (respectively prolongable-premaximal), then there exists $0\leq i\leq n-d+1$ and a strict inclusion $C_i\subset C'_i$ such that $(C_0,\dots,C'_i,\dots,C_n)$ is coherent (respectively prolongable). But since $\vec{C}\prec (C_0,\dots,C'_i,\dots,C_n)$, we obtain a contradiction to coherent-maximality (respectively prolongable-maximality) of $\vec{C}$.
If $\vec{C}$ is not coherent-postmaximal (respectively prolongable-postmaximal), then there exists $n-d+2\leq i\leq n$ and a strict inclusion $C_i\subset C'_i$ such that $(C_{n-2d+3},\dots,C'_i,\dots,C_n)$ is coherent (respectively prolongable). But then every length-$d$ sub-sequence of $\vec{C}':=(C_0,\dots,C'_i,\dots,C_n)$ is coherent, so $\vec{C}'$ is coherent; furthermore, if $(C_{n-2d+3},\dots,C'_i,\dots,C_n)$ is prolongable then so is $\vec{C}'$. This contradicts coherent-maximality (respectively prolongable-maximality) of $\vec{C}$.
\end{proof}

\begin{lem} \label{lem:premaxtopremax}
Let $n\geq 2d-3$. A 
sequence $\vec{C}=(C_0,\dots,C_{n+1})$ is coherent-premaximal (respectively prolongable-premaximal) if and only if:
\begin{enumerate}
    \item $(C_0,\dots,C_n)$ is coherent-premaximal (resp. prolongable-premaximal), and
    \item $C_{n-d+2}$ is maximal among all sets $X\in \mathcal{S}$ such that:
    \[
    (C_{n-2d+3},\dots,C_{n-d+1},X,C_{n-d+3},\dots,C_{n+1})
    \]
    is coherent (respectively prolongable).
\end{enumerate}
\end{lem}

\begin{proof}
For the `if' direction, let $\vec{C}=(C_0,\dots,C_{n+1})$ and assume that conditions (1),(2) hold, and let us prove that $\vec{C}$ is coherent-premaximal (resp. prolongable-premaximal). First, we claim that $\vec{C}$ is coherent. Every length-$d$ sub-sequence of $\vec{C}$ is either a sub-sequence of $(C_0,\dots,C_n)$, hence coherent (since $(C_0,\dots,C_n)$ is coherent by condition (1)), or else must be $(C_{n-d+2},\dots,C_{n+1})$ , which is coherent (respectively prolongable) by condition (2).
Now if $\vec{C}$ is not coherent-premaximal then there exists some $0\leq i\leq n-d+2$ and a strict inclusion such that $C_i\subset C'_i$ such that $(C_0,\dots,C'_i,\dots,C_{n+1})$ is coherent (respectively prolongable). Thus $(C_0,\dots,C'_i,\dots,C_n)$ is coherent (respectively prolongable). If $i\leq n-d+1$ then $(C_0,\dots,C_n)$ is not coherent-premaximal (respectively prolongable-premaximal), contradicting (1). If $i=n-d+2$ then we obtain a contradiction to (2).

For the `only if' direction, assume that $\vec{C}=(C_0,\dots,C_{n+1})$ has the property that it is coherent-premaximal (resp. prolongable-premaximal). In particular, $\vec{C}$ is coherent (resp. prolongable).

Let us prove condition (1). Note that $(C_0,\dots,C_n)$ is coherent (resp. prolongable). If $(C_0,\dots,C_n)$ is not coherent-premaximal (respectively prolongable-premaximal) then there exists some $0\leq i\leq n-d+1$ and a strict inclusion $C_i\subset C'_i$ such that $(C_0,\dots,C'_i,\dots,C_n)$ is coherent (resp. prolongable). We claim that $\vec{C}':=(C_0,\dots,C'_i,\dots,C_{n+1})$ is coherent (resp. prolongable). Every length-$d$ sub-sequence of $\vec{C}'$ is either a sub-sequence of $(C_0,\dots,C'_i,\dots,C_n)$, hence coherent, or else must be $(C_{n-d+2},\dots,C_{n+1})$, which is the length-$d$ suffix sub-sequence of $\vec{C}$ , hence coherent. Furthermore, if $\vec{C}$ is prolongable then so is $\vec{C}'$, since for coherent sequences, prolongability is determined by the length-$(d-1)$ suffix sub-sequence. Therefore $\vec{C}$ is not coherent-premaximal (resp. prolongable-premaximal).

Now let us prove condition (2). Note that $(C_{n-2d+3},\dots,C_{n+1})$ is coherent (resp. prolongable). 
If there exists a strict inclusion $C_{n-d+2}\subset X$ such that:
\[
(C_{n-2d+3},\dots,C_{n-d+1},X,C_{n-d+3},\dots,C_{n+1})
\]
is coherent (resp. prolongable) then every length-$d$ sub-sequence of:
\[
\vec{C}':=(C_0,\dots,C_{n-d+1},X,C_{n-d+3},\dots,C_{n+1})
\] 
is coherent (since it must be a sub-sequence of:
\[
(C_{n-2d+3},\dots,C_{n-d+1},X,C_{n-d+3},\dots,C_{n+1})
\]
or of $(C_0,\dots,C_{n-d+1})$). Furthermore, if: 
\[
(C_{n-2d+3},\dots,C_{n-d+1},X,C_{n-d+3},\dots,C_{n+1})
\]
is prolongable then so is $\vec{C}$, since it they share the same length-$(d-1)$ suffix sub-sequence, and a coherent sequence is prolongable if and only if its length-$(d-1)$ suffix sub-sequence is prolongable. It follows that $\vec{C}\prec \vec{C}'$ are both coherent (respectively prolongable) and share the same length-$(d-1)$ suffix sub-sequence, contradicting the coherent-premaximality (respectively prolongable-premaximality) of $\vec{C}$.
\end{proof}

Let:
\begin{eqnarray*}
\overline{V} & = & \{\vec{C}\in \mathcal{S}^{2d-2}|\vec{C}\ \text{is coherent}\} \\
V & = & \{\vec{C}\in \mathcal{S}^{2d-2}|\vec{C}\ \text{is prolongable}\}
\end{eqnarray*}
Define quivers $\overline{Q}=(\overline{V},\overline{E}),Q=(V,E)$ as follows.
Let us right an arrow in $\overline{Q}$ (resp. $Q$) between elements from $\overline{V}$ (respectively $V$) of the form:
\[
(C_1,\dots,C_{2d-2})\rightarrow (C_2,\dots,C_{2d_1})
\]
if $C_d$ is maximal among all sets $X\in \mathcal{S}$ such that:
\[
(C_{n-2d+3},\dots,C_{n-d+1},X,C_{n-d+3},\dots,C_{n+1})
\] 
is coherent (respectively prolongable).

\begin{proof}[{Proof of Theorem \ref{rational}}]
Let $\vec{C}=(C_0,\dots,C_n)$ for $n\geq 2d-3$ and define: $$\vec{C}^{(i)}:=(C_i,\dots,C_{i+2d-3})$$ for each $0\leq i\leq n-2d+3$. Then by Lemma \ref{lem:maxtopremax} and an inductive application of Lemma \ref{lem:premaxtopremax}, $\vec{C}$ is coherent-maximal (respectively prolongable-maximal) if and only if $\vec{C}^{(0)}$ is coherent-premaximal (respectively prolongable-premaximal) and:
\[
\vec{C}^{(0)}\rightarrow \vec{C}^{(1)}\rightarrow \cdots \rightarrow \vec{C}^{(n-2d+3)}
\]
are arrows in $\overline{Q}$ 
(respectively $Q$), and  $\vec{C}$ is 
coherent-postmaximal (respectively 
prolongable-postmaximal). Notice that 
for a coherent-premaximal 
(respectively prolongable-premaximal) 
sequence, being coherent-postmaximal 
(resp. prolongable-postmaximal) is 
equivalent to having the length-$(2d-
2)$ suffix sub-sequence (for $\vec{C}$ 
this is $\vec{C}^{(n-2d+3)}$) coherent-
postmaximal (resp. prolongable-
postmaximal).

Therefore, $\bar{a}_n$ (resp. $a_n$) is equal to the number of paths of length $n-2d+3$ in $\overline{Q}$ (resp. $Q$) with initial vertex being coherent-premaximal (resp. prolongable-premaximal) and terminal vertex coherent-postmaximal (respectively prolongable-postmaximal). Let:
\[
A_{\overline{Q}}\in M_{|\overline{V}|}(\mathbb{Z}),\ A_{Q}\in M_{|V|}(\mathbb{Z})
\]
be the adjacency matrices of $\overline{Q}$ and $Q$ respectively, namely, 
\[
\left(A_{\overline{Q}}\right)_{v_1,v_2}=\delta_{v_1\rightarrow v_2\in \overline{E}},\ 
\left(A_{Q}\right)_{v_1,v_2}=\delta_{v_1\rightarrow v_2\in E}.
\]
Let:
\[
\overline{w}_{\text{pre}}\in \mathbb{Z}^{1\times |\overline{V}|},\ w_{\text{pre}}\in \mathbb{Z}^{1\times |V|}
\]
be the indicator vectors of coherent-premaximality and prolongable-premaximality in $\overline{Q}$ and $Q$, respectively, and let:
\[
\overline{w}_{\text{post}}\in \mathbb{Z}^{1\times |\overline{V}|},\ w_{\text{post}}\in \mathbb{Z}^{1\times |V|}
\]
be the indicator vectors of coherent-postmaximality and prolongable-postmaximality in $\overline{Q}$ and $Q$, respectively. 
Therefore for every $n\geq 2d-3$:
\[
\bar{a}_n = \overline{w}_{\text{pre}} \cdot A_{\overline{Q}}^{n-2d+3} \cdot \overline{w}_{\text{post}},\ a_n = w_{\text{pre}} \cdot A_{\overline{Q}}^{n-2d+3} \cdot w_{\text{post}}.
\]
It follows that the sequences:
\[
\{\overline{a}_n\}_{n\geq 2d-3},\ \{a_n\}_{n\geq 2d-3}
\]
satisfy linear recurrences, and therefore the generating functions:
\[
\sum_{n=0}^{\infty} \bar{a}_n t^n,\ \sum_{n=0}^{\infty} a_n t^n
\]
are rational.
\end{proof}

\section{Finite-dimensional proalgebraic varieties}

Let $\mathcal{V}=\{(V_n,\varphi_n)\}_{n=1}^{\infty}$ be a proalgebraic variety. By definition, the sequence $\{\dim V_n\}_{n=1}^{\infty}$ is monotone non-decreasing. If this sequence stabilizes, we let $\dim \mathcal{V}=\lim_{n\rightarrow \infty} \dim V_n$, and say that $\mathcal{V}$ is finite-dimensional. Otherwise, we think of $\mathcal{V}$ as an infinite-dimensional object.

\begin{exmpl}
We have $\dim \mathcal{P}(F\left<x_1,\dots,x_d\right>)=\infty$ since $\mathcal{P}_n(F\left<x_1,\dots,x_d\right>)=\left(\mathbb{P}^{d-1}\right)^{\times (n+1)}$.
\end{exmpl}

\begin{exmpl} \label{dim1}
If $A$ is a projectively simple monomial algebra which is not of linear growth then $\dim \mathcal{P}(A)\geq 1$. Indeed, by Proposition \ref{rem}, $A$ admits at least one tree $\mathcal{T}=C_0C_1\cdots$ in which $C_0$ is not a singleton, hence each $\mathcal{P}_n(A)$ projects onto $\mathbb{P}^d$ for some $d\geq 1$, so $\dim \mathcal{P}(A)\geq 1$.
\end{exmpl}

\begin{exmpl}
In Section \ref{sec:Sturmian} we shall see that when $A$ is a projectively simple monomial algebra with Hilbert series $\frac{1}{(1-t)^2}$ then $\mathcal{P}(A)$ is geometrically a union of a projective line with a Cantor set, intersecting at two points; in particular, $\dim \mathcal{P}(A)=1$.
\end{exmpl}

Let us consider a class of monomial algebras (which contains the algebras from the previous example, as will be observed in the sequel) for which $\mathcal{P}(A)$ is finite-dimensional. Let us say that an infinite word $w\in \Sigma^{\mathbb{N}}$ is $k$-\emph{balanced} for some $k\geq 1$ if for every $x\in \Sigma$ and for every equal length factors of $w$, say, $u_1,u_2$, it holds that:
\[
\big||u_1|_x-|u_2|_x\big|\leq k
\] 
where $|u|_x$ is the number of occurrences of the letter $x$ in $u$. Words that are $1$-balanced are called simply \emph{balanced}. 

\begin{prop}
Let $w$ be a $k$-balanced word. Then $\dim \mathcal{P}(A_w)<\infty$.
\end{prop}
\begin{proof}
By Theorem \ref{decomposition}, it suffices to prove that for each tree $\mathcal{T}=C_0C_1\cdots$ over $A_w$, the number of non-singleton $C_i$'s is bounded from above by a constant (which is independent of the tree).
Assume to the contrary that there exists such a tree with non-singletons $C_{i_1},\dots,C_{i_t}$ for some $i_1<\cdots<i_t$ where $t>(k+1)|\Sigma|$ ($\Sigma$ is the underlying alphabet). There exists $x\in \Sigma$ such that $x\in C_{i_j}$ for at least: \[ r:=\left\lfloor \frac{t}{|\Sigma|} \right\rfloor>k \] $j$'s, say, $i_{j_1},\dots,i_{j_r}$. This means that there exist $u_1,\dots,u_{t-1}$ such that $v:=xu_1x\cdots xu_{t-1}x\in C_{i_{j_1}}\cdots C_{i_{j_t}}$. Since each $C_i$ consist of at least two elements, there exist another factor of $\mathcal{T}$ of the form $v':=y_1u_1y_2\cdots y_{t-1}u_{t-1}y_t$ with $y_1,\dots,y_t\in \Sigma\setminus \{x\}$ and clearly $|v|=|v'|$. Now:
\[ \big| |v|_x - |v'|_x \big| = \Bigg| \left(\sum_{i=1}^{t-1} |u_i|_x + t\right) - \left(\sum_{i=1}^{t-1} |u_i|_x \right) \Bigg| = t > k, \]
contradicting that $w$ is $k$-balanced.
\end{proof}

Let $\sigma\colon \Sigma\rightarrow \Sigma^{*}$ be a map such that for some $x\in \Sigma$ the word $\sigma(x)$ starts with $x$, and for every $x\in \Sigma$ the lengths $|\sigma^n(x)|$ tend to infinity; such a map can be extended to a map $\sigma\colon \Sigma^{\mathbb{N}}\rightarrow \Sigma^{\mathbb{N}}$ and is called a \textit{substitution}; consider the $|\Sigma|\times |\Sigma|$ matrix $M_\sigma=(|\sigma(x)|_y)_{x,y}$. If in some power of $M_\sigma$ all entries are positive then $\sigma$ is said to be \textit{primitive}.

Adamczewski \cite[Theorem~13]{Adamcz} proved that if $w\in \Sigma^\mathbb{N}$ is a fixed point (that is, $\sigma(w)=w$) of a primitive substitution and all eigenvalues of $M_\sigma$ except for one have absolute value smaller than $1$, then $w$ is $k$-balanced for some $k$.

\begin{exmpl}
Consider the encoded Thue-Morse sequence via $0\mapsto x,1\mapsto y$: \[ w = xyyxyxxy\dots \]
which is a fixed point of the morphism $\sigma(x)=xy,\sigma(y)=yx$ (the encoding of $\phi$ defining ${\bf t}$ from Subsection \ref{subsec:iso mon alg}). Then: \[ M_\sigma=\left( \begin{matrix} 1 & 1 \\ 1 & 1
    \end{matrix} \right) \] so by the above $w$ is balanced (in fact, $2$-balanced; $w$ is not balanced, equivalently, non-Sturmian; see \cite{Adamcz}). In fact, $\dim \mathcal{P}(A_w)=1$. By Example \ref{dim1}, $\dim \mathcal{P}(A_w)\geq 1$. Assume to the contrary that there was a tree $\mathcal{T}=C_0C_1\cdots$ over $A_w$ with $|C_i|=2$ for some $i>0$; without loss of generality, assume that $C_1=\{x,y\}$ (this can be ensured by shifting each tree such that $C_i$ is placed in $C_1$) and that $C_0=\{y\}$ (if both $x,y\in C_0$ then each letter from $C_2$ would lead to a cube, and the Thue-Morse sequence is cube-free; recall that $A_w$ has an automorphism switching $x$ and $y$).
    Consequently, $C_2=\{x\}$, since the Thue-Morse sequence is cube-free and so $y^3$ is not a factor; next, $C_3=\{y\}$ (since $x^3$ is not a factor) and then $C_4=\{x\}$ (since $y^2 x y^2$ is not a factor of $w$; see \cite{Sgpfrm}, or use the $2$-balanced property). Then $C_5=\{y\}$ (since $ x^2 y x^2$ is not a factor of $w$). But this contradicts that $yxyxy$ does not factor $w$. Hence $\dim \mathcal{P}(A_w)=1$, by Theorem \ref{decomposition}.
\end{exmpl}



\section{Monomial $\mathbb{P}^1$'s} \label{sec:Sturmian}

Fix a binary alphabet $\Sigma = \{x_0,x_1\}$.
We retain the notions from the previous sections. In particular, $p_0=\infty=[1:0], p_1=0=[0:1] \in \mathbb{P}^1$.

\subsection{Monomial $\mathbb{P}^1$'s, Sturmian sequences and their trees}

The homogeneous coordinate ring of $\mathbb{P}^1$ is $F[x_0,x_1]$---a graded integral domain of Hilbert series $H(t)=\frac{1}{(1-t)^2}$.
Let $A$ be a prime monomial algebra with Hilbert series $H_A(t)=\frac{1}{(1-t)^2}=\sum_{n=0}^{\infty} (n+1)t^n$. We refer to such an algebra as a \emph{monomial} $\mathbb{P}^1$.
Every monomial $\mathbb{P}^1$ takes the form $A_w$ for an infinite word $w\in \Sigma^{\mathbb{N}}$ and furthermore, $|\Sigma|=2$ (as $\dim_F A_1=2$). The Hilbert series of $A_w$ translates to the complexity function of $w$ as follows: for every $n\geq 1$, we have $p_w(n)=n+1$. An infinite word $w$ whose complexity function satisfies $p_w(n)=n+1$ is called Sturmian. Sturmian words have been classified by Morse and Hedlund \cite{MH2_1940}. Sturmian sequences are uniformly recurrent, and therefore monomial $\mathbb{P}^1$'s are projectively simple.

For an irrational number $\alpha>0$ and a real number $0\leq r<1$ consider the graph of the linear function $y=\alpha x+r$. This graph intersects the first octant of the integral grid $\mathbb{N}^2$ in a discrete set of points, either on a vertical line or a horizontal line. Moving along this set, from the origin to the right---namely, increasing $x$---we put `$x_0$' for a vertical line intersecting the graph and `$x_1$' for a horizontal line intersecting it.
We let the \textit{cutting sequence} associated with $\alpha,r$ be the resulting sequence:
\[ w=x_{i_0}x_{i_1}\cdots \in \{x_0,x_1\}^{\mathbb{N}} \]
and let $A_{\alpha,r} = A_w$ be the resulting monomial algebra.

\begin{thm}[{Morse-Hedlund, \cite{MH2_1940}; see also \cite{Arnoux}}] 
The following are equivalent for an infinite word $w\in \Sigma^{\mathbb{N}}$:
\begin{itemize}
\item $w$ is Sturmian;
\item $w$ is a cutting sequence;
\item $w$ is balanced and not eventually periodic.
\end{itemize}
\end{thm}


\begin{rem} \label{alpha>1}
The cutting sequence associated with $y=\alpha x+r$ for $\alpha>1$ is obtained from the cutting sequence associated with $y=\frac{1}{\alpha}x+\left(1-\frac{r}{\alpha}\right)$ by switching $x_0\leftrightarrow x_1$, and thus $A_{\alpha,r}\cong A_{\frac{1}{\alpha},1-\frac{r}{\alpha}}$.
\end{rem}

\begin{lem} \label{trees_Sturm}
Let $A$ be a monomial $\mathbb{P}^1$. Then every tree $\mathcal{T}=C_0C_1\cdots$ over $A$ has $|C_0|\leq 2$ and $|C_i|=1$ for all $i\geq 1$, and there is a unique tree $\mathcal{T}=C_0C_1\cdots$ over $A$ with $|C_0|=2$.
\end{lem}
\begin{proof}
Let $w$ be the Sturmian sequence associated with $A$, which can be realized as a cutting sequence for the line $\{y=\alpha x+r\}$ for some irrational $\alpha>0$ and for some $0\leq r<1$.
By Remark \ref{alpha>1} it suffices to prove the claim for $0<\alpha<1$. Notice that $\alpha<1$ implies that $x_1^2$ is not a factor of the associated Sturmian sequence $w$.
By Proposition \ref{rem}, there exists a tree $\mathcal{T}=C_0C_1\cdots$ over $A$ such with $C_0=\{x_0,x_1\}$. We first claim that there is a \textit{unique} such tree. Indeed, we claim that for every $n\geq 1$ there exists a unique length-$n$ factor $u$ of $w$ such that both $x_0u,x_1u$ are factors of $w$. Let us prove this claim by induction; for $n=1$, such $u$ must be equal to $x_0$, since $x_1^2=0$. For the induction step, assume that $u,v$ are length-$(n+1)$ factors of $w$ such that $x_0u,x_1u,x_0v,x_1v$ are all factors of $w$. Write $u=u'x_s,v=v'x_t$ for some $s,t\in \{0,1\}$. By induction, $u'=v'$, since $u',v'$ are length-$n$ factors of $w$ which can ne both prolonged to the left by both $x_0,x_1$. 
Assume to the contrary that $t\neq s$, say, $t=0,s=1$. Thus $x_0u'x_0,x_1u'x_1$ are both factors of $w$ and:
\[
\left||x_0u'x_0|_{x_1} - |x_1u'x_1|_{x_1}\right| = \left||u'|_{x_1} - (|u'|_{x_1}+2)\right| = 2
\]
contradicting that $w$ is balanced. We thus proved that there exists only one tree $\mathcal{T}=C_0C_1\cdots$ with $|C_0|=2$.

Next, we claim that for each tree $\mathcal{T}=C_0C_1\cdots$ over $A$, we have that $|C_i|=1$ for all $i\geq 1$. Fix an arbitrary tree $\mathcal{T}=C_0C_1\cdots$ over $A$. 
Assume to the contrary that $C_i=\{x_0,x_1\}$ for some $i\geq 1$. Then $C_{i-1}=\{x_0\}$, since $x_1^2=0$ in $A$. Since $0<\alpha<1$, there exists some positive integer $d$ such that $\frac{1}{d+1}<\alpha<\frac{1}{d}$. We claim that:
\begin{align}
x_0^{d+2}=0, \label{Sturm_1}
\end{align} for otherwise there were points $(\mu_1,e),(\mu_2,e+1)$ on the line $\{y=\alpha x+r\}$ for some $\mu_2>\mu_1+d+1$, so $\alpha = \frac{1}{\mu_2-\mu_1} < \frac{1}{d+1}$. Furthermore, we claim that:
\begin{align}
x_1x_0^l x_1=0\ \text{for every}\ 0\leq l\leq d-1,    \label{Sturm_2}
\end{align}
for otherwise there were points $(\mu_1,e),(\mu_2,e+1)$ on the line $\{y=\alpha x+r\}$ for some $\mu_2-\mu_1<\mu_1+l+1$, so: 
\[
\alpha = \frac{1}{\mu_2-\mu_1} > \frac{1}{l+1} \geq \frac{1}{d}.\]

It now follows that $C_{i+1}=\cdots=C_{i+d}=\{x_0\}$, since if there was $1\leq j\leq d$ minimal such that $x_1\in C_{i+j}$ then  $x_1x_0^{j-1}x_1\in C_i\cdots C_{i+j}$, contradicting (\ref{Sturm_2}). Recall that $C_{i-1}=\{x_0\}$ and that $x_0\in C_i$, so now $x_0^{d+2}\in C_{i-1}\cdots C_{i+d}$, contradicting (\ref{Sturm_1}). To conclude, we proved that every tree $\mathcal{T}=C_0C_1\cdots$ over $A$ has $|C_0|\leq 2$ and $|C_i|=1$ for all $i\geq 1$, and that there is a unique tree over $A$ with $|C_0|=2$.
\end{proof}

\subsection{The geometry of monomial $\mathbb{P}^1$'s}

\begin{thm} \label{Sturmian_iso}
Let $A$ be a monomial $\mathbb{P}^1$. Then $\mathcal{P}(A)$ is a union of a projective line with a Cantor set with exactly two intersection points. In particular, all monomial $\mathbb{P}^1$'s have isomorphic proalgebraic varieties of point modules.
\end{thm}

A consequence of Theorem \ref{Sturmian_iso} is:

\begin{cor}
There exists a continuum of pairwise non-isomorphic projectively simple monomial algebras of quadratic growth having equal Hilbert series and pairwise isomorphic proalgebraic varieties of point modules.
\end{cor}

For instance, all of the algebras $\{A_{\alpha,0}\}_{\alpha\in (0,1)\setminus \mathbb{Q}}$ are pairwise non-isomorphic: by Theorem \ref{iso mon alg}, if $A_{\alpha,0}\cong A_{\beta,0}$ for $\alpha,\beta$ as above then there is a monomial isomorphism between them. It can be seen that $x_0^2\neq 0,x_1^2=0$ in both algebras, and therefore a monomial isomorphism would induce the identity map; but the cutting sequences obtained from $\alpha,\beta$ are distinct, and the corresponding monomial algebras can be separated by a factor of the form $x_1x_0^dx_1$, as seen in the proof of Lemma \ref{trees_Sturm}.

\begin{proof}[{Proof of Theorem \ref{Sturmian_iso}}]
Let $A$ be a monomial $\mathbb{P}^1$. By Lemma \ref{trees_Sturm} there exists a unique tree over $A$ with $C_0=\{x_0,x_1\}$ and $|C_i|=1$ for all $i\geq 1$, and all of the other trees over $A$ are singletons, namely, $|C_i|=1$ for all $i\geq 0$.
Let $\mathcal{C}$ denote the set $\{ (p_{i_0},p_{i_1},\dots)|x_{i_0}x_{i_1}\cdots\in \Subshift(A) \}\subseteq \{p_0,p_1\}^{\mathbb{N}}$.
By Theorem \ref{decomposition}, 
\begin{align}
\varprojlim \mathcal{P}(A)=\left(\mathbb{P}^1\times \{p_{\mu_1}\} \times \{p_{\mu_2}\}\times \cdots \right) \cup \mathcal{C}, \label{10.6 (1)}
\end{align}
a union of a projective line, corresponding to the unique tree with $|C_0|=2, |C_{\geq 1}|=2$, with a continuum of points corresponding to $\Subshift(A)$. We denote by $\pi_n$ the projection $\varprojlim\mathcal{P}_n(A)\twoheadrightarrow \mathcal{P}_n(A) \subseteq \left(\mathbb{P}^1\right)^{\times (n+1)}$, and by abuse of notation we also use it to denote the projection $\pi_N(A)\rightarrow \mathcal{P}_n(A)$ for every $N\geq n$.
Notice that the pro-Zariski topology on $\mathcal{C}$ coincides with the product topology on $\Subshift(A)$ as a subspace of $\Sigma^{\mathbb{N}}$, and $\Subshift(A)$ is homeomorphic to the Cantor set by Lemma \ref{NoIsoPt}. There are two points in the first set in the union (\ref{10.6 (1)}) which belong to $\{p_0,p_1\}^\mathbb{N}$, namely:
\begin{eqnarray*}
P_0 & := & \{p_0\}\times \{p_{\mu_1}\} \times \{p_{\mu_2}\}\times\cdots \\
P_1 & := & \{p_1\}\times \{p_{\mu_1}\} \times \{p_{\mu_2}\}\times\cdots.
\end{eqnarray*}

Consider another monomial $\mathbb{P}^1$, say, $B$ and let
$\mathcal{D}=\{ (p_{i_0},p_{i_1},\dots)|x_{i_0}x_{i_1}\cdots\in \Subshift(B) \}\subseteq \{p_0,p_1\}^{\mathbb{N}}$ and:
\[ \varprojlim \mathcal{P}(B)=\left(\mathbb{P}^1\times \{p_{\nu_1}\} \times \{p_{\nu_2}\}\times \cdots \right) \cup \mathcal{D} \]
and define:
\begin{eqnarray*}
Q_0 & := & \{p_0\}\times \{p_{\nu_1}\} \times \{p_{\nu_2}\}\times\cdots \\
Q_1 & := & \{p_1\}\times \{p_{\nu_1}\} \times \{p_{\nu_2}\}\times\cdots
\end{eqnarray*}
and by abuse of notation, we denote by $\pi_n$ the projection $\varprojlim\mathcal{P}_n(B)\twoheadrightarrow \mathcal{P}_n(B)\subseteq \left(\mathbb{P}^1\right)^{\times (n+1)}$.

We now claim that $\mathcal{P}(A)\cong \mathcal{P}(B)$.
Notice that both $\mathcal{C}$ and $\mathcal{D}$ are homeomorphic to the Cantor set, hence there exists a homeomorphism $f\colon \mathcal{C}\rightarrow \mathcal{D}$. Moreover, we may take such homeomorphism for which $f(P_0)=Q_0,f(P_1)=Q_1$. Indeed, we can split $\mathcal{C}$ into a disjoint union of clopen sets $\mathcal{C}=\mathcal{C}_0\cup \mathcal{C}_1$ with $P_i\in \mathcal{C}_i$ ($i=0,1$) and similarly $\mathcal{D}=\mathcal{D}_0\cup \mathcal{D}_1$ with $Q_i\in D_i$ ($i=0,1$), and each $\mathcal{C}_i,\mathcal{D}_i$ is homeomorphic to the Cantor set itself. Since the Cantor set is homogeneous (that is, it is transitive under the action of its auto-homeomorphism group), being homeomorphic to a topological group $\left(\mathbb{Z}/2\mathbb{Z}\right)^{\mathbb{N}}$, we can find homeomorphisms $\mathcal{C}_i\rightarrow \mathcal{D}_i$ carrying $P_i\mapsto Q_i$ and glue them together to $f\colon \mathcal{C}\rightarrow \mathcal{D}$.


By the definition of the projective limit topology, it follows that for every $\xi \in \pi_n(\mathcal{D})$ there exists some $r=r(\xi)$ and $u_0,\dots,u_r\in \{0,1\}$ such that: 
\[ (\pi_n\circ f)^{-1}(\xi) \supseteq \{(p_{u_0},\dots,p_{u_r})\}\times \{p_0,p_1\}^{\times \infty}. \] 
Let: \[ g(n) > \max \{r(\xi)|\xi \in \pi_n(\mathcal{D})\}. \] 
Then $\pi_n \circ f$ is constant on the fibers of $\pi_{g(n)}\colon \mathcal{C}\rightarrow \{p_0,p_1\}^{g(n)+1}$ and thus restricts to a map $f_n\colon \pi_{g(n)}(\mathcal{C})\rightarrow \pi_n(\mathcal{D})$ obtained by $f_n=\pi_n \circ f \circ \theta_{g(n)}$ where $\theta_i\colon \pi_i(\mathcal{C})\rightarrow \mathcal{C}$ is an arbitrary section of the projection $\pi_i\colon \mathcal{C}\rightarrow \{p_0,p_1\}^{i+1}$:
\[
\begin{tikzcd}
\mathcal{C} \arrow[r, "f"]                         & \mathcal{D} \arrow[d, "\pi_n"] \\
\pi_{g(n)}(\mathcal{C}) \arrow[u, "\theta_{g(n)}"] \arrow[r, "f_n"]  & \pi_n(\mathcal{D}).                             \end{tikzcd}
\]
We can now define a family of morphisms of projective varieties:
\[ \{ \psi_n\colon \mathcal{P}_{g(n)}(A)\rightarrow \mathcal{P}_n(B)\ |\ n=1,2,\dots\} \]
as follows. Observe that:
\begin{eqnarray*}
\mathcal{P}_n(A) & = & \left( \mathbb{P}^1\times \{(p_{\mu_1},\dots,p_{\mu_n})\} \right) \cup \pi_n(\mathcal{C}) \\ 
\mathcal{P}_n(B) & = & \left( \mathbb{P}^1\times \{(p_{\nu_1},\dots,p_{\nu_n})\} \right) \cup \pi_n(\mathcal{D})
\end{eqnarray*}
and let $\psi_n\colon \mathcal{P}_{g(n)}(A)\rightarrow \mathcal{P}_n(B)$ be:
\begin{equation*}
\psi_n(\rho)=
    \begin{cases}
        f_n(\rho) & \text{if } \rho \in \pi_{g(n)}(\mathcal{C}) \\
        (\alpha,p_{\nu_1},\dots,p_{\nu_n}) & \text{if } \rho = (\alpha,p_{\mu_1},\dots,p_{\mu{g(n)}})\   \text{for some } \alpha\in \mathbb{P}^1,
    \end{cases}
\end{equation*}
which is a well-defined morphism since: 
\begin{eqnarray*}
f_n(p_0,p_{\mu_1},\dots,p_{\mu_n}) =  (\pi_n\circ f)(P_0)=\pi_n(Q_0)=(p_0,p_{\nu_1},\dots,p_{\nu_n}), \\
f_n(p_1,p_{\mu_1},\dots,p_{\mu_n}) = (\pi_n\circ f)(P_1)=\pi_n(Q_1)=(p_1,p_{\nu_1},\dots,p_{\nu_n}).
\end{eqnarray*}
In a similar vein, we can define $\psi'_n\colon \mathcal{P}_{h(n)}(B)\rightarrow \mathcal{P}_n(A)$ utilizing the inverse isomorphism $f^{-1}\colon \mathcal{D}\rightarrow \mathcal{C}$.
Thus: \[ \Psi = \{ \psi_n\colon \mathcal{P}_{g(n)}(A)\rightarrow \mathcal{P}_n(B)\ |\ n=0,1,2,\dots\} \] provides a morphism of proalgebraic varieties $\mathcal{P}(A)\rightarrow \mathcal{P}(B)$, which is  easily seen to be an isomorphism, whose inverse is the morphism:
\[ \Psi^{-1} = \{ \psi'_n\colon \mathcal{P}_{h(n)}(B)\rightarrow \mathcal{P}_n(A)\ |\ n=0,1,2,\dots\}, \]
completing the proof.
\end{proof}

\subsection{Zhang rigidity}

By Theorem \ref{Sturmian_iso} all monomial $\mathbb{P}^1$'s have isomorphic proalgebraic varieties of point modules, despite being pairwise non-isomorphic to each other. Since graded algebras which are isomorphic to twists of each other have isomorphic proalgebraic varieties of point modules (Proposition \ref{ZhangTwistProvarieties}), it is natural to ask what the twists of monomial $\mathbb{P}^1$'s can be. It turns out that monomial $\mathbb{P}^1$'s are very rigid in this context.

\begin{prop} \label{ZhangRigidity}
Let $A$ be a monomial $\mathbb{P}^1$. If $B$ is a prime monomial algebra which is isomorphic to a twist of $A$ then $A\cong B$.
\end{prop}

Let us first develop a useful combinatorial machinery. Let $\Sigma$ be a finite alphabet,  $A=F\left<\Sigma\right>/I$ a monomial algebra and let $\Mon(A)$ be the directed (rooted, layered) tree whose vertices are non-zero monomials of $A$ (including $1$), with a vertex $u\rightarrow v$ if $v=ux_i$ for some letter (that is, a monomial generator) $x_i\in \Sigma$. Then $A$ is prolongable if and only if $\Mon(A)$ has no sinks. For each $i\geq 0$ let $\Mon_i(A)$ denote the $i$-th layer of the graph, namely, the set of vertices labeled by monomials of length $i$.

\begin{exmpl}
Let $w=yxyyxyxyyx\cdots$ be the Fibonacci word, which is the Sturmian word whose first letter is $y$ and is fixed by the morphism $\phi(x)=y$, $\phi(y)=yx$. As a cutting sequence, $w$ is obtained from the line $\{y=\frac{1}{1.618\dots}x\}$. 
The following figure shows the subgraph of $\Mon(A_w)$ induced by the vertices $\Mon_{\leq 4}(A_w)$:
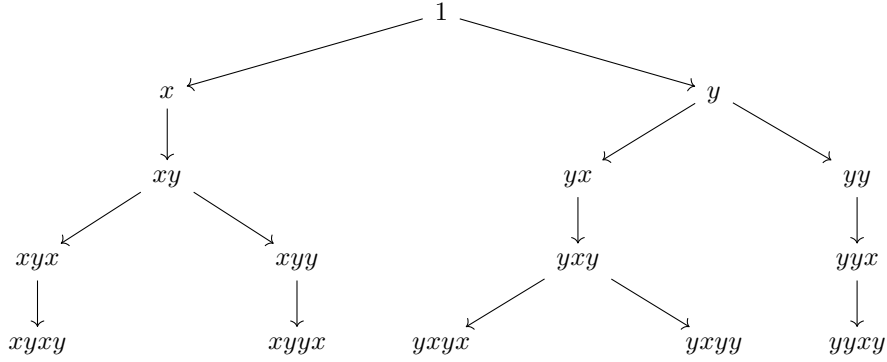
\[
\begin{tikzcd}
              &                          &               & 1 \arrow[lld] \arrow[rrd] &                           &                         &               \\
              & x \arrow[d]              &               &                           &                           & y \arrow[ld] \arrow[rd] &               \\
              & xy \arrow[ld] \arrow[rd] &               &                           & yx \arrow[d]              &                         & yy \arrow[d]  \\
xyx \arrow[d] &                          & xyy \arrow[d] &                           & yxy \arrow[ld] \arrow[rd] &                         & yyx \arrow[d] \\
xyxy          &                          & xyyx          & yxyx                      &                           & yxyy                    & yyxy         
\end{tikzcd}
\]

\end{exmpl}

\bigskip

\begin{prop} \label{Sturmian_graph}
Let $A,B$ be monomial $\mathbb{P}^1$'s such that $\Mon(A)\cong \Mon(B)$. Then $A\cong B$.
\end{prop}
\begin{proof}
Write $A=F\left<x,y\right>/I,\ B=F\left<x,y\right>/J$ and suppose that $f\colon \Mon(A)\rightarrow \Mon(B)$ is a directed graph isomorphism. Assume to the contrary that $A\not \cong B$. 
Since $A,B$ are monomial $\mathbb{P}^1$'s, each layer $\Mon_d(A),\Mon_d(B)$ consists of exactly $d+1$ vertices, and there are no sinks (that is, vertices with no outgoing arrow).
We may assume without loss of generality that $x^2\in I,J$ (in a Sturmian sequence the square of one of the letters is not a factor), so $x$ has only one child in both $\Mon(A),\Mon(B)$. 
Since $\dim_F A_2=\dim_F B_2=3$, it follows that $y^2,yx,xy\notin I,J$, so $y$ has two children in both $\Mon(A),\Mon(B)$, and therefore the subgraph $\Mon_x(A)$ of $\Mon(A)$ consisting of $x$ and its descendants must be mapped under $f$ to the corresponding one in $\Mon(B)$, call it $\Mon_x(B)$.

Since for every $d\geq 0$ we have $\# \Mon_d(A)=\# \Mon_d(B)=d+1$ and both $\Mon(A),\Mon(B)$ have no sinks, in each layer there is exactly one vertex splitting (namely, having two children) and the rest are non-splitting (have only one child each). Since $A\neq B$, there is some minimum index $m$ for which $\Mon_m(A)\neq \Mon(B)_m$ (as sets); such $m$ exists, for otherwise, $\Mon(A)=\Mon(B)$ and so $A=B$. We may assume that $m\geq 3$ by the previous arguments.
In particular, $\Mon_i(A)=\Mon_i(B)$ for all $i\leq m-1$ and consequently, the subgraphs induced by the vertices $\Mon_{\leq m-1}(A)$ and $\Mon_{\leq m-1}(B)$ are equal; in particular, the same vertex splits in $\Mon_{m-2}(A)$ and in $\Mon_{m-2}(B)$.

\bigskip

\textit{Claim}: The splitting vertices in $\Mon_{m-1}(A)$ and $\Mon_{m-1}(B)$ have distinct ancestors from $\Mon_1(A)=\Mon_1(B)=\{x,y\}$; in other words, they lie in distinct halves of $\Mon_{\geq 1}(A),\Mon_{\geq 1}(B)$, respectively. 

\textit{Proof of Claim}. There is a unique vertex $v\in \Mon_{m-2}(A)=\Mon_{m-2}(B)$ with two children in both $A,B$. 
Therefore, the only vertices in $\Mon_{m-1}(A)=\Mon_{m-1}(B)$ that could possibly split in either $\Mon(A)$ or $\Mon(B)$ are $xv$ and $yv$ (indeed, a suffix of a splitting monomial is splitting). Assume to the contrary that $A,B$ have the same splitting vertex in their $(m-1)$-th layer, say, $ev$ for some $e\in \{x,y\}$. Consider an arbitrary vertex $z_1\cdots z_{m-1}\in \Mon_{m-1}(A)=\Mon_{m-1}(B)$ which is non-splitting in both $A,B$ (that is, distinct from $ev$). We claim that it has the same child in both $A,B$; otherwise, if $z_1\cdots z_{m-1}z'\in \Mon_m(A)$ and $z_1\cdots z_{m-1}z''\in \Mon_m(B)$ for $\{z',z''\}=\{x,y\}$ then $z_2\cdots z_{m-1}$ is splitting in both $A,B$ and thus $z_2\cdots z_{m-1}=v$. It follows that $z_1\neq e$, since $z_1v=z_1\cdots z_{m-1}$ is non-splitting. Since $z_1v$ has different children in $A,B$, in one of these algebras (say, $C\in \{A,B\}$) we have that $z_1vz_1\in \Mon_m(C)$; but also, $eve\in \Mon_m(C)$, since $ev$ is assumed to split in both $A,B$. Thus:
\[
\left||eve|_e - |z_1vz_1|_e \right| = 2,
\]
contradicting that $C$ is a monomial $\mathbb{P}^1$ and hence its underlying infinite word is balanced.
We thus proved that $\Mon_m(A)=\Mon_m(B)$: they have a common splitting vertex in $\Mon_{m-1}(A)=\Mon_{m-1}(B)$, and every non-splitting vertex in $\Mon_{m-1}(A)=\Mon_{m-1}(B)$ has the same child in both $\Mon_m(A)$ and $\Mon_m(B)$. This contradicts the definition of $m$. It follows that the assumption that $A,B$ have the same splitting vertex in their $(m-1)$-th layer is wrong, and the claim is proved.

\bigskip

Hence in one of $A,B$ the vertex labeled $xv$ is splitting whereas in the other one $yv$ is splitting. Recall that the isomorphism $f$ carries $\Mon_x(A)$ isomorphically onto $\Mon_x(B)$. But in one of $\Mon_x(A),\Mon_x(B)$ there is a splitting vertex in layer $m-2$ (that is, layer $m-1$ in $\Mon(A)$ or $\Mon(A)$ respectively), whereas in the other one there is no splitting vertex in the same layer. We obtain a contradiction to the initial assumption that $A\not \cong B$, so $A\cong B$.
\end{proof}

\begin{exmpl}
Consider the (right) prolongable monomial algebras: \[ 
 A=F\left<x,y,z\right>/\left<x^2,y^2,z^2\right>,\ B=F\left<x,y,z\right>/\left<xz,yz,z^2\right>. \] Then $\Mon(A)\cong \Mon(B)$, both having a root of degree $3$ and all other layers having two descendants for each vertex, but $A$ is not isomorphic to a twist of $B$, since $A$ is prime and in particular also left prolongable, while $B$ is not left prolongable (see Lemma \ref{Zhang_prolongable}).
\end{exmpl}

\begin{proof}[{Proof of Proposition \ref{ZhangRigidity}}]
By assumption, $B\cong A^\tau$ for some twisting system $\tau$ of $A$, through a degree-preserving isomorphism. This isomorphism induces an isomorphism $\mathcal{P}(B)\cong \mathcal{P}(A)$ which is shift-equivariant, namely, a system of isomorphism of projective varieties $f_n\colon \mathcal{P}_n(B)\xrightarrow{\sim} \mathcal{P}_n(A)$ such that $f_{n-1}\circ \pi_n=\pi_n\circ f_n$ (Proposition \ref{ZhangTwistProvarieties}). Since twists preserve Hilbert series, it follows that $B\cong A^\tau$ is a monomial $\mathbb{P}^1$ too.


Now, by Theorem \ref{Sturmian_iso} we have that:
\begin{eqnarray*}
\mathcal{P}_n(A) & = & \left( \mathbb{P}^1\times \{(p_{\mu_1},\dots,p_{\mu_n})\} \right) \cup \pi_n(\mathcal{C}) \\ 
\mathcal{P}_n(B) & = & \left( \mathbb{P}^1\times \{(p_{\nu_1},\dots,p_{\nu_n})\} \right) \cup \pi_n(\mathcal{D}),
\end{eqnarray*}
where:
\begin{eqnarray*}
\mathcal{C} & = & \{ (p_{i_0},p_{i_1},\dots)|x_{i_0}x_{i_1}\cdots\in \Subshift(A) \}\subseteq \{p_0,p_1\}^{\mathbb{N}} \\
\mathcal{D} & = & \{ (p_{i_0},p_{i_1},\dots)|x_{i_0}x_{i_1}\cdots\in \Subshift(B) \}\subseteq \{p_0,p_1\}^{\mathbb{N}}
\end{eqnarray*}
and the copies of $\mathbb{P}^1$ intersect $\pi_n(\mathcal{C})$ and $\pi_n(\mathcal{D})$ at exactly two points (each).

For each $n$, the isomorphism $f_n$ carries the copy of $\mathbb{P}^1$ in $\mathcal{P}_n(B)$ to the copy of $\mathbb{P}^1$ in $\mathcal{P}_n(A)$, and consequently, $f_n(\pi_n(\mathcal{D} \setminus \mathbb{P}^1))=\pi_n(\mathcal{C} \setminus \mathbb{P}^1)$. 

Every non-zero monomial in $A,B$ has at least two (possibly long) prolongations (since the the underlying infinite word is recurrent and no eventually periodic), so $\pi_n(\mathcal{C})=\pi_n(\mathcal{C}\setminus \mathbb{P}^1)$, and therefore $f_n$ carries $\pi_n(\mathcal{D})$ to $\pi_n(\mathcal{C})$; denote the induced isomorphism (of finite sets!) by $\bar{f}_n$.
Since the system $\{f_n\}_{n=1}^{\infty}$ (and hence $\{\bar{f}_n\}_{n=1}^{\infty}$) is shift-equivariant, we obtain an isomorphism of directed graphs $\Mon(B)\xrightarrow{\sim} \Mon(A)$. By Proposition \ref{Sturmian_graph}, it follows that $A\cong B$.
\end{proof}

\end{document}